\documentclass[10pt,a4paper]{article}
\usepackage[T1]{fontenc}
\usepackage[left=2cm, right=2cm, top=3cm, bottom=3cm]{geometry}
\usepackage{todonotes}
\usepackage{xcolor}
\definecolor{internationalkleinblue}{rgb}{0.0, 0.18, 0.65}
\providecommand{\keywords}[1]{\textbf{Keywords: } #1}
\usepackage{amssymb,amsthm,amsmath,dsfont}
\numberwithin{equation}{section}
\usepackage{mathrsfs}
\usepackage{natbib,hyperref}
\usepackage{cleveref}
\usepackage{authblk}
\usepackage{enumerate}
\newcommand{\scr}[1]{\mathscr{#1}}
\newcommand{\bb}[1]{\mathbb{#1}}
\newcommand{\al}[1]{\mathcal{#1}}
\newcommand{\rrm}[1]{\mathrm{#1}}
\newcommand{\I}{\mathds{1}}
\newtheorem{theorem}{Theorem}[section]
\newtheorem{definition}[theorem]{Definition}
\newtheorem{proposition}[theorem]{Proposition}
\newtheorem{lemma}[theorem]{Lemma}
\newtheorem{corollary}[theorem]{Corollary}
\newtheorem{remark}[theorem]{Remark}
\newtheorem{assumption}[theorem]{Assumption}
\newtheorem{example}[theorem]{Example}

\title{Long-time behaviour of DLRA for SDEs}
\author[a,1]{Jianhai Bao}
\author[b,2]{Haitao Wang}
\author[b,3]{Yue Wu}
\affil[a]{Center for Applied Mathematics, Tianjin University, 300072 Tianjin, P.R. China}
\affil[b]{Department of Mathematics and Statistics, University of Strathclyde, Glasgow, G1 1XH, UK} \affil[1]{\href{mailto:jianhaibao@tju.edu.cn}{jianhaibao@tju.edu.cn}}
\affil[2]{\href{mailto:haitao.wang@strath.ac.uk}{haitao.wang@strath.ac.uk}}
\affil[3]{\href{mailto:yue.wu@strath.ac.uk}{yue.wu@strath.ac.uk}}
\date{}

\begin{document}
	\maketitle

\begin{abstract}
    We study dynamical orthogonal (DO) approximations of stochastic differential equations and investigate their long-time behaviour. The DO formulation represents the solution by a low-rank decomposition and leads to a coupled system consisting of an evolution equation on the Stiefel manifold and a reduced stochastic process. We establish the well-posedness of the strong DO system and derive quantitative error estimates between the original stochastic differential equation and its low-rank approximation in the Wasserstein distance.
    
Our main contribution is the analysis of invariant probability measures for the DO dynamics. Under suitable dissipativity, Lipschitz continuity, and non-degeneracy assumptions on the coefficients, we prove the existence of an invariant probability measure for the strong DO system. The proof combines uniform moment estimates, a Krylov--Bogoliubov argument for an associated frozen system, and a Kakutani-Fan-Glicksberg fixed-point theorem to recover the self-consistent dynamics. We further show that the induced low-rank process admits an invariant probability measure and discuss the structure of invariant measures through several illustrative examples. These results provide a rigorous foundation for the use of dynamical low-rank approximations in the approximation of long-time statistical properties of stochastic dynamical systems.\\
\end{abstract}
	\section{Introduction}\label{sec.intr}
High-dimensional stochastic differential equations (SDEs) play a central role both as models, for example in molecular dynamics \cite{niknam2024nuclear}, Bayesian inverse problems \cite{blomker2022continuous}, uncertainty quantification \cite{guilleminot2014ito}, and stochastic filtering \cite{llopis2018particle}, and as the continuous-time foundations of sampling algorithms such as Langevin-type Markov chain Monte Carlo \cite{basak2016langevin}. 
In many of these settings, the objective is not merely to approximate trajectories accurately over a finite time horizon, but to characterise long-time statistical properties of the system, such as invariant measures, ergodic averages, and equilibrium fluctuations. 
In high state-space dimension $d$, direct numerical simulation of SDEs can become prohibitively expensive, especially when accurate approximation of long-time behavior requires integration over long time horizons. This motivates the development of structure-preserving model reduction techniques that reduce computational cost while retaining key dynamical, statistical, or geometric properties of the original high-dimensional SDE.

Model reduction techniques aim to approximate the solution of a high-dimensional SDE by a lower-dimensional surrogate, thereby reducing the cost of simulation while maintaining accuracy for the quantities of interest. Among these approaches, Dynamical Low-Rank Approximation (DLRA) has emerged as a flexible and computationally efficient framework. Its key advantage is that the reduced subspace is not fixed in advance; rather, it evolves dynamically with the solution. As a result, DLRA can adapt to time-dependent coherent structures, transient dynamics, and changing directions of stochastic variability, while retaining a compact representation of the evolving solution.

The dynamical low-rank approximation was originally developed for deterministic matrix differential equations by Koch and Lubich \cite{koch2007dynamical}. Independently, Sapsis and Lermusiaux derived the Dynamically Orthogonal (DO) field equations for stochastic systems \cite{sapsis2009dynamically}. The connection between the two frameworks was later pointed out in \cite{musharbash2015error}. This framework has been applied to random partial differential equations \cite{feppon2018dynamically,kazashi2021existence,musharbash2018dual} and rigorously analyzed in the stochastic setting by Kazashi, Nobile, and Zoccolan \cite{DLRASDE}.

To describe the main idea, consider a $d$-dimensional stochastic differential equation (SDE) on a filtered complete probability space $\big(\Omega,\scr F,(\scr F_t)_{t\ge0},\bb P\big)$
	\begin{equation}\label{eq.sec4}
		\begin{aligned}
		&\rrm dX_t=b(t,X_t)\,\rrm dt+\sigma(t,X_t)\,\rrm dB_t,\\&
		X_0=\xi \in L^2(\Omega,\bb R^d),
		\end{aligned}
	\end{equation}
	where  $(B_t)_{t\ge0}$ is an $m$-dimensional Brownian motion, the functions $b:[0,\infty)\times\bb R^d\to \bb R^d$ and $\sigma:[0,\infty)\times\bb R^d\to\bb R^{d\times m}$ satisfy suitable assumptions specified below in Section \ref{sec:properties}. 
    The DLRA method approximates the solution $X_t \in \mathbb R^d$
of Eqn. \eqref{eq.sec4} by a rank-$r$ ($r\le d$) representation of the form, termed as the \emph{DLR approximation},
\begin{equation}\label{eqn:intro_xr}
    X_t^{(r)} = U_t^\top Y_t=\sum_{i=1}^{r}u_{t,i}y_{t,i}\in \bb R^d,
\end{equation}
	where,
the pair $(U_t=(u_{t,i})_{i=1}^{r},Y_t=((y_{t,i})_{i=1}^r):[0,T]\to \bb R^{r\times d}\times L^2(\Omega,\bb R^r)$ 
is called a \emph{strong DO solution} for Eqn. \eqref{eq.sec4}, if it satisfies 
		\begin{itemize}
			\item[(1)] the initial data $(U_0,Y_0)$ are chosen so that $U_0\in\bb R^{r\times d}$ has orthonormal rows and  the components of $Y_0$ are linearly independent in $L^2(\Omega)$.
			\item[(2)] the function $U:[0,T]\to\bb R^{r\times d}$ is absolutely continuous and satisfies
		\begin{equation}\label{eq:U}
			U_t U_t^\top = I_{r} \in \mathbb{R}^{r \times r} \text{ and } \qquad
			U_t \frac{\partial U_t^\top}{\partial t} = O_{r} \in \mathbb{R}^{r \times r} \quad \text{for almost every } t \in [0,T],
		\end{equation}
		where $I_r$ and $O_r$ denote the identity and zero matrices in $\bb R^{r\times r}$, respectively.
			\item[(3)] the function $Y(\omega):[0,T]\to\bb R^r$ is almost surely continuous paths and it is $\scr F_t$-measurable for all $t\in[0,T]$. Besides, for any $t\in[0,T]$, $y_{t,1},y_{t,2},\cdots,y_{t,r}$ are linearly independent in $L^2(\Omega)$.
			\item[(4)] for almost every $t\in[0,T]$, the pair $(U,Y):[0,T]\to \bb R^{r\times d}\times L^2(\Omega,\bb R^r)$  satisfies the equation 
				\begin{equation}\label{eq:DO}
				\begin{aligned}
					C_{Y_t}\rrm dU_t&=\bb E[Y_tb(t,U_t^\top Y_t)^\top](I_{d}-P_{U_t})\,\rrm dt;\\
					\rrm d Y_t&=U_tb(t,U_t^\top Y_t)\,\rrm dt+U_t\sigma(t,U_t^\top Y_t)\,\rrm dB_t,
				\end{aligned}
			\end{equation}
			where $C_{Y_t}$ is the covariance matrix of $Y_t$, i.e. $C_{Y_t}=\bb E[Y_tY_t^\top]$ and $P_{U_t}$ is the orthogonal projection matrix onto the row space  span$\{u_{t,1},u_{t,2},\cdots,u_{t,r}\}\subset \bb R^d$ that means $P_{U_t}=U_t^\top U_t$. 
		\end{itemize}
        
	Note that conditions in Eqn. \eqref{eq:U} imply that the rows of $U_t$  remain orthonormal over time \cite{sapsis2009dynamically} thus the operator $P_{U_t}$ is an orthogonal projection. As a consequence, the \emph{DLR} approximation admits the following SDE 
	\begin{equation}\label{eqn:DLRA}
			\rrm d X^{(r)}_{t}
			=\big((I_{d}-P_{U_t})\big(P_{Y_t}b(t,X^{(r)}_{t})\big)+P_{U_t}b(t,X^{(r)}_{t})\big)\,\rrm dt+P_{U_t}\sigma(t,X^{(r)}_{t})\,\rrm d B_t,
	\end{equation}
	where $P_{Y_t}$ is the projection from $L^2(\Omega,\bb R^d)$ to span$\{y_{t,1},y_{t,2},\cdots,y_{t,r}\}$, i.e. $P_{Y_t}u=\bb E[uY_t^\top]C_{Y_t}^{-1}Y_t$.

The DLRA method described above is consistent at full rank (see Proposition \ref{PropXandXDO}) in the sense that, when $r=d$, the DLRA manifold coincides with the full ambient space 
$\mathbb R^d$, and the reduced system formally recovers the original SDE 
(up to an orthogonal change of coordinates), independently of whether 
the diffusion is degenerate or the law has full-dimensional support.

Existing work has focused primarily on
well-posedness \cite{DLRASDE},
numerical aspects of the method \cite{kazashi2026numerical,ueckermann2013numerical} and applications in Kalman filtering \cite{nobile2026dynamicallowrankensemblekalman,nobile2025dynamicallowrankapproximationskalman}. 
In many applications, the primary object of interest is approximating the law of the process $X_t$ in the long-time regime. 
Typical questions concern the existence and characterization of invariant measures, convergence rates toward equilibrium, and the preservation of ergodic structure under approximation. From this perspective, a reduced-order model is meaningful only if it faithfully captures, at least qualitatively and ideally quantitatively, the invariant or stationary behaviour of the original dynamics.
This leads to the following fundamental question:
\begin{quote}
\emph{Does the DLR approximation preserve the long-time
distributional properties of the original SDE?}
\end{quote}

Addressing this question is substantially more delicate than in the classical setting of finite-dimensional Markov diffusions. Assume that the $d$-dimensional Markovian SDE \eqref{eq.sec4} admits an invariant probability measure $\mu$, for instance under suitable dissipativity of the drift and nondegeneracy of the diffusion. For the original SDE, the long-time distributional behavior can often be studied through the Markov semigroup associated with $X_t$, using tools such as Lyapunov estimates, irreducibility, strong Feller properties, coupling arguments, or hypoelliptic criteria.

In the DLRA framework, however, the reduced process has a more intricate structure. The approximation $X_t^{(r)}=U_t^\top Y_t$ does not, by itself, define an autonomous Markov diffusion on $\mathbb R^d$. Instead, the evolution is described by the coupled system \eqref{eq:DO} for the time-dependent basis $U_t$ and the stochastic coordinates $Y_t$. Moreover, the coefficients in \eqref{eq:DO} contain expectations such as $\bb E[Y_tb(t,U_t^\top Y_t)^\top]$ and the covariance matrix $C_{Y_t}$, and hence depend implicitly on the law of the reduced variables. In this sense, the DO/DLRA system has a McKean--Vlasov-type character: the evolution of each component depends not only on its current state, but also on distributional quantities generated by the approximation itself. Consequently, classical invariant-measure arguments for Markov semigroups on a fixed finite-dimensional state space are not directly applicable.

A further difficulty comes from the geometry of the approximation. The factor $U_t$ evolves on the Stiefel manifold of orthonormal $r$-frames in $\mathbb R^d$, while $Y_t$ satisfies an $r$-dimensional SDE whose coefficients depend on $U_t$. 
Since $X_t^{(r)} = U_t^\top Y_t$ with a time-varying basis $U_t$, this moving-subspace structure complicates the identification of stationary regimes: even if the law of $Y_t$ becomes stationary in reduced coordinates, the induced law of $X_t^{(r)}$ in $\mathbb R^d$ may still depend on the limiting behavior of the basis $U_t$. It also obstructs the direct use of standard hypoelliptic irreducibility criteria based on bracket-generating conditions, since the relevant dynamics do not form a standard autonomous diffusion on the full space $\mathbb R^d$ \cite{charous2023dynamically,DLRASDE,musharbash2015error}.

There is also an unavoidable approximation-theoretic obstruction. If the original SDE is nondegenerate, for example uniformly elliptic, then its invariant measure $\mu$ typically has full support in $\mathbb R^d$. By contrast, at each time the DLR approximation is supported on the rank-$r$ subspace determined by $U_t$. Therefore, unless the invariant distribution of the original SDE is itself intrinsically low-dimensional, one cannot expect the DLR approximation to recover $\mu$ exactly when $r<d$. The appropriate goal is instead to quantify the discrepancy between the invariant law of the original dynamics and the long-time distribution generated by the reduced dynamics in a weak metric, such as the Wasserstein distance.

Thus, preservation of long-time distributional properties under DLRA is neither automatic nor structurally guaranteed. It requires a probabilistic analysis of the nonlinear reduced system, uniform-in-time estimates that prevent the reduced dynamics from escaping to infinity, and a precise formulation of what it means for the invariant law of the reduced system to approximate that of the original SDE. The present work addresses these issues by studying the asymptotic behavior of the DO/DLRA dynamics and by quantifying, under suitable structural assumptions, the approximation of the invariant law of the original SDE.

Our main contributions are as follows.

\begin{itemize}
\item[(1)] \emph{Well-posedness under weakened assumptions.}
We establish well-posedness of the DO/DLRA system under the weakened condition stated in Assumption~\ref{as:bo1}. Compared with the assumptions imposed in \cite{DLRASDE}, this relaxes the conditions required for the existence of DLRA solutions to the SDE \eqref{eq.sec4}; see Theorem~\ref{wellposedness}.

\item[(2)] \emph{Uniform-in-time moment bounds.}
Under suitable dissipativity assumptions, we prove that both the exact solution and its DLR approximation admit second-moment bounds that are uniform in time. These estimates, developed in Section~\ref{sec:uniform}, provide the compactness and stability input needed for the subsequent long-time analysis.

\item[(3)] \emph{Wasserstein stability of the DLR approximation.}
In Theorem~\ref{maintheorem}, we derive quantitative bounds on the 2-Wasserstein distance between the law $\mu_t$ of the exact solution $X_t$ and the law $\nu_t$ of the DLR approximation $X_t^{(r)}$. Under global dissipativity, we show that
\begin{equation}
W_2^2(\mu_t,\nu_t)
\le
e^{\delta t} W_2^2(\mu_0,\nu_0)
+
C(d-r),
\qquad \delta<0,
\end{equation}
where the residual term $C(d-r)$ quantifies the error induced by the rank constraint. In particular, the DLRA approximation inherits the exponential contraction mechanism of the original SDE up to an irreducible rank-defect error.

\item[(4)] \emph{Sharpness of the rank-defect error.}
For a class of Ornstein--Uhlenbeck dynamics, we compute the Wasserstein distance explicitly and show that the residual term proportional to $d-r$ is optimal. This demonstrates that the error term in the general stability estimate is not merely an artifact of the proof; see Example~\ref{exam:sharp} in Section \ref{ssec:sharpness}.

\item[(5)] \emph{Existence of invariant measures for the reduced dynamics.}
Under suitable coercivity and ellipticity conditions, we prove that the DLRA system admits an invariant probability measure; see Theorem~\ref{inv}. The proof combines uniform moment bounds, tightness obtained through Krylov--Bogoliubov averaging, and a Kakutani fixed-point argument to handle the nonlinear, law-dependent structure of the reduced dynamics.
\end{itemize}

The paper is organised in the following way:
Section \ref{sec:pre} introduces notation and recalls preliminaries. Section \ref{sec:properties} establishes well-posedness of the DLRA system.
Section \ref{sec:uniform} proves uniform-in-time moment bounds
and Wasserstein stability estimates.
Section \ref{sec:invariant} establishes existence of invariant measures.
	
    \section{Preliminaries}\label{sec:pre}
    In this section, we introduce the notation used throughout the paper and recall several results that play an important role in our analysis. 

\paragraph{Notations.} Let $\bb N$ denote the set of all positive integers. For $d\in\bb N$, we write $|\cdot|$ for the Euclidean norm on Euclidean space $\bb R^d$ associated with the inner product $\langle \cdot,\cdot\rangle$. For $m,d\in\bb N$, the notation $\|\cdot\|_\rrm{HS}$ refers to the Hilbert-Schmidt norm for linear maps from $\bb R^m$ to $\bb R^d$. This norm is induced by the Hilbert–Schmidt inner product $\langle A,B\rangle_\rrm{HS}=\rrm{Tr}(A^\top B)$ and coincides with the $\ell^2$-norm of the matrix entries, i.e., $\|A\|_\rrm{HS}=\sqrt{\sum_{i=1}^d\sum_{j=1}^ma_{i,j}^2}$, where $A=(a_{i,j})$.
	 This norm coincides with the standard matrix norm on $\bb R^{m\times d}$ induced by the Euclidean norm. We denote by $O(r)$ the set of all orthogonal matrices in $\bb R^{r\times r}$. For any $a,b\in \bb R$, we define $$a\wedge b=\min \{a,b\},\text{ and } a\vee b=\max\{a,b\}.$$
    For any matrices  $A,\,B\in\bb R^{d\times d}$ with $d\in \bb N$, we write $A\succeq B$ (resp. $A\succ B$) to denote that $A-B$ is positive semidefinite (resp. positive definite). Moreover, we denote by $C_{\mathrm{b}}(\bb R^d)$ the set of all bounded continuous functions on $\bb R^d$.

    \paragraph{Function spaces and measure-theoretic framework.} 
    Let $V$ be a topological space with $\scr B(V)$ being its Borel $\sigma$-algebra. Denote by $\scr P(V)$  the set of all probability measures on $(V,\scr B(V))$. 

    For $V=\mathbb{R}^d$, a probability measure $\mu\in \scr P(\bb R^d)$  is called \emph{tight} if for any $\varepsilon>0$, there is a compact subset $K_\varepsilon$ of $\bb R^d$ such that $\mu(\bb R^d\setminus K_\varepsilon)\le \varepsilon$. For $p\ge 1$, we write $\scr P_p(\bb R^d)\subset \scr P(\bb R^d)$ for the subset of probability measures with a finite $p$-th moment, that is,
    \begin{equation}
        \scr P_{p}(\bb R^d)=\Big\{\mu\in\scr P(\bb R^d):\mu(|\cdot|^p):=\int_{\bb R^d}|x|^p\,\mu(\rrm dx)\Big\}.
    \end{equation}
    The space $\scr P_{p}(\bb R^d)$ can be equipped with the $p$-Wasserstein distance defined by
	$$\bb W_p(\mu,\nu):=\inf_{\pi\in\Pi(\mu,\nu)}\Big(\int_{\bb R^d\times\bb R^d}|x-y|^p\,\pi(\rrm dx,\rrm dy)\Big)^\frac{1}{p},$$ 
	where $\Pi(\mu,\nu)$ denotes the set of all couplings of $\mu$ and $\nu$, that is, all joint probability measures on $\bb R^d\times\bb R^d$ with marginals $\mu$ and $\nu$ respectively. More precisely, a measure $\pi\in\Pi(\mu,\nu)$ satisfies $\pi(\bb R^d,\rrm dy)=\nu(\rrm dx)$ and $\pi(\rrm dx,\bb R^d)=\mu(\rrm dx)$~\cite{coupling}. Let $\xi,\zeta$ be random variable on space $(\bb R^d,\scr B(\bb R^d))$ such that $\scr L_\xi=\mu$ and $\scr L_\zeta=\nu$, where $\scr L_\xi$ denotes the law of $\xi$. Then it holds that $\bb W^p_p(\mu,\nu)\le \bb E[|\xi-\zeta|^p]$. For any $\mu$-integrable function $f:\bb R^d\to\bb R$, we use the notation 
    \begin{equation}\label{eqn:int}
    	\mu(f):=\int_{\bb R^d}f(x)\,\mu(\rrm dx).
    \end{equation}
  
 \paragraph{Random variable spaces.}
Let $(\Omega,\scr F, (\scr F_t)_{t\ge0},\bb P)$ be a filtered complete probability space.
For $p\ge1$ and $d\in\bb N$, we denote by 
\[
L^p(\Omega,\bb R^d)
:= \Big\{ \xi:\Omega\to\bb R^d \text{ measurable }:\,
\bb E|\xi|^p<\infty \Big\},
\]
equipped with the norm
\[
\|\xi\|_{L^p(\Omega,\bb R^d)}
:= \big(\bb E[|\xi|^p]\big)^{1/p}.
\]
When $p=2$, this is a Hilbert space with inner product
\[
\langle \xi,\zeta\rangle_{L^2(\Omega,\bb R^d)}
:= \bb E\big[\xi^\top\zeta\big].
\]
For $\xi\in L^1(\Omega,\bb R^d)$, we denote its law by
\[
\scr L_\xi := \bb P\circ \xi^{-1} \in \scr P(\bb R^d).
\]
Moreover, $\xi\in L^p(\Omega,\bb R^d)$ if and only if
$\scr L_\xi \in \scr P_p(\bb R^d)$.
Also, for $\scr F_s\in(\scr F_t)_{t\ge0}$,  we define the expectation of the random variable $\xi$ under $\scr F_s$ by
\begin{equation}\label{eqn:conditional}
    \bb E_s[\xi] = \bb E[\xi|\scr F_s].
\end{equation}
    \section{Well-posedness and some properties of DLRA and DO solution}\label{sec:properties}
	This section is devoted to the existence of the DLRA over an infinite time horizon and to {the relation between the SDE~\eqref{eq.sec4} and its DLRA~\eqref{eqn:intro_xr}.}
    We consider stochastic differential equations (SDEs) of the form Eqn.~\eqref{eq.sec4} and impose the following assumptions on the drift and diffusion coefficients. {ompared with \cite{DLRASDE}, we relax the global Lipschitz condition on the drift, but we require a non‑degeneracy condition on the diffusion coefficients.}

	\begin{assumption}\label{as:bo1}
		The drift coefficient $b:[0,\infty)\times\bb R^d\to \bb R^d$ and the diffusion coefficient $\sigma:[0,\infty)\times \bb R^d\to\bb R^{d\times m}$ are measurable and $1/2$-H\"older continuous in time, uniformly in $x\in\bb R^d$.
		\begin{enumerate}[(H1)]
			\item For any open set $E \subset \bb R^d$, there exists a constant $L_{b,1}(E) \ge 0$ such that, for all $t \in [0,\infty)$ and $x, y \in E$, 
			\begin{align*}
				|b(t,x)-b(t,y)|&\le L_{b,1}(E)|x-y|.
			\end{align*}
			\item There exists a constant $L_{b,2} \ge 0$ such that, for all $t \in [0,\infty)$ and $x \in \bb R^d$,
			\begin{align*}
				|b(t,x)|&\le L_{b,2}(|x|+1).
			\end{align*}
			\item There exists a constant $L_{\sigma,1}\ge0$ such that, for all $t\in[0,\infty)$ and $x,y\in\bb R^d$,
			$$\left\|\sigma(t,x)-\sigma(t,y) \right\|_\rrm{HS}^2\le L_{\sigma,1}|x-y|^2.$$
			\item There exists $L_{\sigma,2}\ge0$ such that, for all $t\in[0,\infty)$ and $x\in\bb R^d$,
			\begin{align*}
				\|\sigma(t,x)\|^2_\rrm{HS}&\le L_{\sigma,2}(|x|^2+1).
			\end{align*}
            \item There exists a constant $L_{\sigma,3}>0$ such that, for all $t\in[0,\infty)$ and $x\in\bb R^d$,
            \begin{align*}
                \sigma(t,x)\sigma(t,x)^\top\succeq L_{\sigma,3}I_{d}.
            \end{align*}
		\end{enumerate}
	\end{assumption}	
The following theorem establishes the global existence and uniqueness of the strong DO solution, the maximal existence time being characterized by the loss of invertibility of the covariance matrix.
	\begin{proposition}
	    \label{wellposedness}
		Suppose Assumptions~(H1)--(H5) hold and recall $C_{Y_t}$ defined in Eqn.~\eqref{eq:DO}, and assume that there exists $\sigma_{Y_0}>0$ such that $\bb E[Y_0Y^\top_0]\succeq\sigma_{Y_0}I_{r}$. Then,  for any $T>0$, the SDE~\eqref{eq.sec4} has a unique strong DO solution on $(0,T]$.
	\end{proposition}
	\begin{proof}
		For each $n\in\bb N^+$, define a truncation function $\phi_n:\bb R^d\to \bb R^d$ given by 
		\begin{equation}
			\phi_n(x)=\left\lbrace\begin{aligned}
				&x\, \text{ when }|x|\le n,\\
				&\frac{x}{|x|}n\,\text{ when }|x|>n.
			\end{aligned}\right. 
		\end{equation}
		The function $\phi_n:\bb R^d\to\bb R^d$ is globally Lipschitz continuous; more precisely, 
		\begin{equation}\label{eq.phi}
			|\phi_n(x)-\phi_n(y)|\le |x-y|, \, \text{for all }x,\, y\in \bb R^d.
		\end{equation}
		Given $n\in\bb N^+$, define the truncated drift function $b_n:[0,\infty)\times \bb R^d\to \bb R^d$ by $$b_n(t,x)=b(t,\phi_n(x)),\text{ for all } x \in\bb R^d.$$
		By Assumption~(H1) and Eqn.~\eqref{eq.phi}, the function $b_n:[0,\infty)\times\bb R^d\to \bb R^d$ is globally Lipschitz continuous. Besides, since the drift function $b$ satisfies a linear growth condition~(H2), for any $x\in\bb R^d$, we have
		\begin{equation}\label{eqn:grobn}
			|b_n(t,x)|=|b(t,\phi_n(x))|\le L_{b,2}(1+|\phi_n(x)|)\le L_{b,2}(1+|x|),
		\end{equation}
		where $L_{b,2}>0$ is the linear growth constant, independent of $n$.
		
		Now, for each $n\in\bb N^+$, consider the following dynamically orthogonal equation in $\bb R^d$
		\begin{align}\label{eq.yu}
			\begin{split}
				C_{Y_{t,n}}\rrm dU_{t,n}&=\bb E[Y_{t,n}b_n(t,U_{t,n}^\top Y_{t,n})^\top](I_{d}-P_{U_{t,n}})\,\rrm dt,\\
				\rrm d Y_{t,n}&=U_{t,n}b_n(t,U_{t,n}^\top Y_{t,n})\,\rrm dt+U_{t,n}\sigma(t,U_{t,n}^\top Y_{t,n})\,\rrm dB_t.
			\end{split}
		\end{align}
		According to~\cite[Theorem 4.6]{DLRASDE}, Eqn.~\eqref{eq.yu} admits a unique solution $(U_{t,n},Y_{t,n})$ for all $t\ge0$.  The proof of~\cite[Proposition~4.5]{DLRASDE} relies only on the linear growth of $b_n$ and $\sigma$ (see Eqn.~\eqref{eqn:grobn} and condition~(H4)), on condition~(H5), and on the existence of $\sigma_{Y_0}>0$ such that $\mathbb{E}[Y_0Y_0^\top]\succeq\sigma_{Y_0}I_r$. All these conditions are independent of $n$; therefore the same proof directly yields a constant $\sigma_Y>0$, also independent of $n$, such that
        \begin{equation}\label{eqn:semiposi}
            C_{Y_{t,n}}\succeq \sigma_YI_r.
        \end{equation}

        For any $n\in\bb N^+$, define a stopping time $\tau_n$ by 
		\begin{equation}\label{eq.tau}
			\tau_n=\inf\{t\ge0,|Y_{t,n}|>n\}.	
		\end{equation}
		We now divide the proof into four parts to establish the desired result.
		\\
		\textbf{I. Prove $Y_{t\wedge\tau_n,n+1}=Y_{t\wedge\tau_n,n}$ and $U_{t\wedge\tau_n,n+1}=U_{t\wedge\tau_n,n}$ a.s. with respect to $\bb P$ for all $t\ge0,\, n\in\bb N^+$.}
		
		Observe that $|U_{s,n}^\top Y_{s,n}|\le n$ for all $s\in[0,\tau_n)$, since $U_{s,n}$ is orthonormal and $|Y_{s,n}|\le n$ by definition of $\tau_n$. Therefore, the truncation function satisfies $$\phi_{n+1}(U_{s,n}^\top Y_{s,n})=\phi_n(U_{s,n}^\top Y_{s,n}),$$ and hence
		\begin{equation}
			b_{n+1}(t,U_{s,n}^\top Y_{s,n})=b_n(t,U_{s,n}^\top Y_{s,n}),\, \text{for all }s\in[0,\tau_n).
		\end{equation}
		By Eqn.~\eqref{eq.yu} and  H\"older's inequality, we estimate the difference between the two solutions 
		\begin{align*}
			|Y_{t\wedge\tau_n,n+1}-Y_{t\wedge\tau_n,n}|^2&\le 2t\int_{0}^{t\wedge\tau_n}\big|U_{s,n+1}b_{n+1}(s,U_{s,n+1}^\top Y_{s,n+1})-U_{s,n}b_{n+1}(s,U_{s,n}^\top Y_{s,n})\big|^2\,\rrm ds\\&\quad			+2\Big|\int_{0}^{t\wedge\tau_n}\big(U_{s,n+1}\sigma(t,U^\top_{s,n+1}Y_{s,n+1})+U_{s,n}\sigma(t,U^\top_{s,n}Y_{s,n})\big)\,\rrm dB_s\Big|^2.
		\end{align*}
		Applying the Burkholder-Davis-Gundy (BDG) inequality, the global Lipschitz continuity of $b_n$, (H3) and taking expectations, we obtain the existence of a constant $C_n>0$ depending on $n$ such that
		\begin{align*}
			\bb E\big[\sup_{0\le s\le t}|Y_{s\wedge\tau_n,n+1}-Y_{s\wedge\tau_n,n}|^2\big]&\le 2t\bb E\Big[\sup_{0\le r\le t}\int_{0}^{r\wedge\tau_n}\big|U_{s,n+1}b_{n+1}(s,U_{s,n+1}^\top Y_{s,n+1})-U_{s,n}b_{n+1}(s,U_{s,n}^\top Y_{s,n})\big|^2\,\rrm ds\Big]\\&\quad 
			+2\bb E\Big[\sup_{0\le r\le t}\Big|\int_{0}^{r\wedge\tau_n}\big(U_{s,n+1}\sigma(t,U^\top_{s,n+1}Y_{s,n+1})+U_{s,n}\sigma(t,U^\top_{s,n}Y_{s,n})\big)\,\rrm dB_s\Big|^2\Big]\\&\le 
			C_n\int_{0}^{t}\Big(\bb E\big[\sup_{0\le r\le s}|Y_{r\wedge\tau_n,n+1}-Y_{r\wedge\tau_n,n}|^2\big]+\bb E\big[\sup_{0\le r\le s}\|U_{r\wedge\tau_n,n+1}-U_{r\wedge\tau_n,n}\|^2_\rrm{HS}\big]\Big)\rrm ds.    
		\end{align*}
		Since Assumption (H5) holds, \cite[Proposition 4.5]{DLRASDE} implies that both $C_{Y_{t,n}}$ and $C_{Y_{t,n+1}}$ are invertible.  Therefore, by~\cite[Lemma A.4]{DLRASDE}, we obtain the estimate
		\begin{align*}
		\bb E\big[\sup_{0\le s\le t}\|U_{s\wedge\tau_n,n+1}-U_{s\wedge\tau_n,n}\|^2_\rrm{HS}\big]\le C_n\int_{0}^{t}\Big(\bb E\big[\sup_{0\le r\le s}|Y_{r\wedge\tau_n,n+1}-Y_{r\wedge\tau_n,n}|^2\big]+\bb E\big[\sup_{0\le r\le s}\|U_{r\wedge\tau_n,n+1}-U_{r\wedge\tau_n,n}\|^2_\rrm{HS}\big]\Big)\rrm ds.
		\end{align*}
		Combining the two inequalities and applying Gr\"onwall's inequality, we conclude that 
		$$\bb E\big[\sup_{0\le s\le t}|Y_{s\wedge\tau_n,n+1}-Y_{s\wedge\tau_n,n}|^2\big]+\bb E\big[\sup_{0\le s\le t}\|U_{r\wedge\tau_n,n+1}-U_{r\wedge\tau_n,n}\|^2_\rrm{HS}\big]=0.$$
		Therefore, for all $n\in\bb N^+$, we have
        \begin{equation}\label{eq:same}
            Y_{t\wedge\tau_n,n+1}=Y_{t\wedge\tau_n,n}, \text{ and } U_{t\wedge\tau_n,n+1}=U_{t\wedge\tau_n,n}\,\text{ for all $t\ge0\ $ a.s}.
        \end{equation}
		
		From Eqn.~\eqref{eq:same} and the definition of $\tau_n$ in Eqn.~\eqref{eq.tau}, it follows that the sequence $(\tau_n)_{n\in\bb N^+}$ is non-decreasing; that is,  $$\tau_{n+1}\ge\tau_n,\quad \text{for all }n\in\bb N.$$
		\textbf{II. Proof that $\bb E[\, \sup_{0\le s\le t}|Y_{s,n}|^2]$ is bounded.}
		Recall that for any $x\in\bb R^d$
		\begin{equation*}
			|b_n(t,x)|\le L_{b,2}(1+|x|).
		\end{equation*}

		Using (H4) together with Eqn.~\eqref{eqn:grobn}, and applying H\"older's inequality and the BDG inequality, we obtain a constant $C_{t,b,\sigma} > 0$ independent of $n$ such that
		\begin{align*}
			\bb E\big[\sup_{0\le s\le t}|Y_{s,n}|^2\big]&\le \bb E[|Y_{0,n}|^2]+2t\bb E\Big[\sup_{0\le s\le t}\int_0^s|U_{r,n}b_n(r,U_{r,n}^\top Y_{r,n})|^2\,\rrm dr\Big]\\&\quad+2\bb E\Big[\sup_{0\le s\le t}\Big|\int_0^sU_{r,n}\sigma(r,U_{r,n}^\top Y_{r,n})\,\rrm dB_r\Big|^2\Big]
			\\&\le \bb E[|Y_0|^2]+C_{t,b,\sigma}\int_{0}^{t}\bb E\big[\sup_{0\le r\le s}|Y_{r,n}|^2\big]\,\rrm ds.
		\end{align*}
		Applying Gr\"onwall's inequality yields the desired bound
		\begin{equation}\label{eq.Yn}
			\bb E\big[\sup_{0\le s\le t}|Y_{s,n}|^2\big]\le C_{t,b,\sigma}\bb E[|Y_0|^2].
		\end{equation}
		
		Since the stopping times $\tau_n$ are non-decreasing almost surely, we may define the limiting stopping time $\tau_\infty$ by 
		$$\tau_\infty=\lim_{n\to\infty}\tau_n.$$
		\textbf{III. Proof that $\tau_\infty=+\infty$ almost surely.}\\
		From the definition of $\tau_n$ in Eqn.~\eqref{eq.tau}, we observe that for all $\widetilde T> 0$,
		\begin{equation}
			\{\tau_n<\widetilde T\}\subset\big\lbrace \sup_{0\le t\le\widetilde T}|Y_{s,n}|\ge n\big\rbrace.
		\end{equation}
		Applying Chebyshev’s inequality and the bound from Eqn.~\eqref{eq.Yn}, we obtain
		\begin{align*}
			\bb P(\tau_n\le\widetilde T)\le \bb P\big(\sup_{0\le t\le\widetilde T}|Y_{s,n}|\ge n\big)\le \frac{1}{n^2}\bb E\big[\sup_{0\le t\le\widetilde T}|Y_{s,n}|^2\big]\le \frac{C_{\widetilde T,b,\sigma}\bb E[|Y_0|^2]}{n^2},
		\end{align*}
		which implies 
		$$\sum_{n=0}^{\infty}	\bb P(\tau_n\le\widetilde T)\le \sum_{n=0}^{\infty}\frac{C_{\widetilde T,b,\sigma}\bb E[|Y_0|^2]}{n^2}<\infty.$$
		By the Borel-Cantelli lemma, we conclude
		$$\bb P\big( \bigcap_{i=1}^{\infty} \bigcup_{n=i}^{\infty}\{\tau_n\le\widetilde T\}\big)=0.$$
		Thus, there exists a $\bb P$-null set $\Omega_n\subset\Omega$ with $\bb P(\Omega_n)=0$ such that for every $\omega\in \Omega\setminus \Omega_n$, we can find an integer $N_0=N_0(\omega)$ satisfying $$\tau_n(\omega)\ge\widetilde T,\, \text{for all } n\ge N_0.$$ 
		This implies 
		$$\tau_\infty(\omega)=\lim_{n\to \infty}\tau_n(\omega)\ge\widetilde T,\, \text{for all }\omega\in \Omega\setminus \Omega_n.$$
		Let $\Omega_0:=\bigcup_{n=1}^\infty\Omega_n$. We have $\bb P(\Omega_0)=0$, and for all $\omega\in \Omega\setminus \Omega_0$, $$\tau_\infty(\omega)\ge T,\, \text{for all } T> 0,$$
		which implies $\tau_\infty=+\infty$ almost surely.
		\\
		\textbf{IV. Construct the global solution  $(U,Y)$.}
		
		For each fixed $n\in\bb N^+$ and any $T\ge0$, let $(U_{\cdot,n},Y_{\cdot,n})$ be the unique solution to Eqn.~\eqref{eq.yu}. From Part I and II, there exists $\Omega_0\subset \Omega$ with $\bb P(\Omega_0)=0$ such that for all $n\in\bb N^+$, $\omega\in\Omega\setminus \Omega_0$ and $t\ge0$,  $$\tau_\infty(\omega)=\lim_{n\to\infty}\tau_n(\omega)=+\infty,\ Y_{t\wedge\tau_n,n+1}=Y_{t\wedge\tau_n,n},\,\text{and }U_{t\wedge\tau_{n},n+1}=U_{t\wedge\tau_{n},n}.$$
		This shows that for every $\omega\in\Omega\setminus\Omega_0$ and $t\ge0$, there exists $N_0=N_0(\omega)>0$ such that
		\begin{equation}\label{eq.yut}
		 Y_{t,n}=Y_{t,N_0},\ U_{t,n}=U_{t,N_0},\text{ and }\tau_n>t,\ \forall~n\ge N_0.
		\end{equation}
		Define mappings $U:[0,T]\to \bb R^{r\times d}$ and $Y:[0,T]\times \Omega\to L^2(\Omega,\bb R^r)$ by 
		\begin{equation}\label{eq.map}
			U_t:=\lim_{n\to\infty}U_{t,n},
			\quad 
			Y_t=\left\lbrace\begin{aligned}
				&Y_{t,n}, \text{ when }\omega\in\Omega\setminus\Omega_0,\,  t\in[0,\tau_n];\\
				&Y_{0},\text{ when }\omega\in\Omega_0,\,  t\ge0.
			\end{aligned} \right. 
		\end{equation}
		By Eqn.~\eqref{eq.yut}, for any $\omega\in\Omega\setminus\Omega_0$, there exists $N_1=N_1(\omega)$ such that \begin{equation}\label{eqn:relation}
		    \tau_n(\omega)\ge t,\, \text{and }  Y_{t,n}=Y_{t,N_1},\, U_{t,n}=U_{t,N_1}, \text{ for all }  n\ge N_1.
		\end{equation}
		Hence, for any $m>0$, 
		\begin{align*}
			\|U_{t,m+n}-U_{t,n}\|^2_\rrm{HS}&=\bb E\big[\|U_{t,m+n}-U_{t,n}\|^2_\rrm{HS}\,(\I_{\{t>\tau_n\}}+\I_{\{t\le\tau_n\}})\big]\\&=\bb E\big[\|U_{t,m+n}-U_{t,n}\|^2_\rrm{HS}\,\I_{\{t>\tau_n\}}\big]+\bb E\big[\|U_{t,m+n}-U_{t,n}\|^2_\rrm{HS}\,\I_{\{t\le\tau_n\}}\big]\\&= \|U_{t,m+n}-U_{t,n}\|^2_\rrm{HS}\,\bb P(\tau_n<t)\\&\le \|U_{t,m+n}-U_{t,n}\|^2_\rrm{HS}\,\bb P(\Omega_0)\\&\le 0.
		\end{align*}
		This confirms that the mappings $U,\, Y$ defined in Eqns.~\eqref{eq.map} are well-defined.  Moreover, since each $U_{\cdot,n}$ is absolutely continuous and each $Y_{\cdot,n}$ is almost surely continuous, it follows from Eqns.~\eqref{eq.Yn} and~\eqref{eq.map} that $U_\cdot$ is absolutely continuous and $Y_\cdot$ is almost surely continuous. For each fixed $t\ge0$, we have
		$$\lim_{n\to \infty}Y_{t,n}=Y_t, \text{ for all }\omega\in \Omega\setminus\Omega_0,$$
		where we know that $Y_t$ is $\scr F_t$-adapted. By Eqns.~\eqref{eq.Yn},~\eqref{eq.map} as well as Fatou's Lemma, we have 
        \begin{equation}
            \bb E[|Y_t|^2]\le \liminf_{n\to\infty}\bb E[|Y_{t,n}^2]\le C_{t,b,\sigma}\bb E[|Y_0|^2],
        \end{equation}
        where $C_{t,b,\sigma}$ is the same number as Eqn.~\eqref{eq.Yn}.
        
        Since $\Omega_0$ is a null set, for every $n$ we have 
        \begin{equation}\label{eqn.as}
        |Y_{t,n}-Y_{t}|^2 \I_{\Omega_0} = 0 \qquad \text{almost surely}. 
        \end{equation}
        Moreover, by Eqns.~\eqref{eq.map} and~\eqref{eqn:relation}, for every $\omega \in \Omega\setminus\Omega_0$, there exists $N_2 = N_2(\omega)$ such that $\tau_{N_2}(\omega) \ge T$, and
        \begin{equation}\label{eqn.equa}
            Y_{t,n}=Y_{t,N_2}=Y_{t},\text{ for all }n\ge N_2,\,t\in[0,T].
        \end{equation}
        Combining Eqn.~\eqref{eqn.as} and Eqn.~\eqref{eqn.equa}, we obtain, for every $t\in[0,T]$,
        \begin{equation}\label{eqn:L2conver}
             \lim_{n\to\infty}\bb E[|Y_{t,n}-Y_{t}|^2]=\lim_{n\to\infty}\bb E[|Y_{t,n}-Y_{t}|^2\I_{\Omega\setminus\Omega_0}]+\lim_{n\to\infty}\bb E[|Y_{t,n}-Y_{t}|^2\I_{\Omega_0}]=0.
        \end{equation}
		 
		Next, we show that $(U,Y)$ is a solution of Eqn.~\eqref{eq:DO}. Since $(U_{\cdot,n},Y_{\cdot,n})$ is a solution of Eqn.~\eqref{eq.yu}, we have 
		 \begin{align}\label{eq.yy}
         Y_{t\wedge\tau_n,n}=Y_{0}+\int_0^{t\wedge\tau_n}U_{s,n}b_n(s,U_{s,n}^\top Y_{s,n})\,\rrm ds+\int_0^{t\wedge\tau_n}U_{s,n}\sigma(s,U_{s,n}^\top Y_{s,n})\,\rrm dB_s.
		 \end{align}
		 By Eqns.~\eqref{eq.map}, we have $Y_{t\wedge\tau_n,n}=	Y_{t\wedge\tau_n}$ almost surely. Moreover, by the definitions of $b_n$, for any $x\in\bb R^d$, we get 
		 $$|b_n(t,x)-b(t,x)|=|b(t,\phi_n(x))-b(t,x)|\le C_x|\phi_n(x)-x|\to0\, \text{ as }n\to\infty.$$
		 Taking $n\to+\infty$, we get from Eqn.~\eqref{eq.yy},
		 \begin{equation}
		 	Y_{t}=Y_{0}+\int_0^{t}U_{s}b(s,U_{s}^\top Y_{s})\,\rrm ds+\int_0^{t}U_{s}\sigma(s,U_{s}^\top Y_{s})\,\rrm dB_s.
		 \end{equation}
		 On the other hand, from Eqns.~\eqref{eq.Yn} and~\eqref{eqn:L2conver} we deduce that
		 \begin{equation}
		 		\lim_{n\to \infty}C_{Y_{t,n}}=C_{Y_{t}}.
		 \end{equation} 
     From \eqref{eqn:semiposi} and \cite[Theorem 3.3.16(c)]{horn2012matrix} we obtain $C_{Y_{t}}\succeq\sigma_YI_r$.
		 Moreover, by the dominated convergence theorem, we obtain
		 \begin{align*}
		 	\lim_{n\to\infty}\bb E\big[Y_{t,n}b_n(t,U_{t,n}^\top Y_{t,n})^\top\big]=\bb E\big[\lim_{n\to\infty}Y_{t,n}b_n(t,U_{t,n}^\top Y_{t,n})^\top\big]=\bb E\big[Y_{t}b(t,U_{t}^\top Y_{t})^\top\big].
		 \end{align*}
		 Taking $n\to+\infty$ from Eqn.~\eqref{eq.yu}, we deduce that the limiting process $U_t$ satisfies
		 \begin{equation}
		 	U_{t}=U_0+\int_0^t C^{-1}_{Y_s}\bb E\big[Y_{s}b(t,U_{s}^\top Y_{s})^\top\big](I_{d}-P_{U_{s}})\,\rrm ds.
		 \end{equation}
		This confirms that the pair $(U,Y)$ constructed above is indeed a global solution to the original system \eqref{eq:DO}.\\
		\textbf{V. Uniqueness of the Solution $(U,Y)$.}
		
		Suppose $(U,Y)$ and $(V,Z)$ are two solutions of Eqn.~\eqref{eq:DO} with the same data $(U_0,Y_0)$.  For each $n\in\bb N^+$ and fixed $T>0$, define a stopping time $\tau_n$ by
		\begin{equation}
			\tau_n=\inf\{t\ge0:|Y_t|\ge n\text{ or }|Z_t|\ge n\}\wedge T.
		\end{equation}
		Using the same method as in Part I (e.g., H\"older's inequality, BDG's inequality and Lipschitz continuity of $b$), we obtain the following estimates for all $t\in[0,T]$
		\begin{equation}
			\bb E\big[\sup_{0\le s\le t}|Y_{s\wedge\tau_n}-Z_{s\wedge\tau_n}|^2\big]\le C\int_{0}^{t}\Big(\bb E\big[\sup_{0\le r\le s}|Y_{r\wedge\tau_n}-Z_{r\wedge\tau_n}|^2\big]+\bb E\big[\sup_{0\le r\le s}\|U_{r\wedge\tau_n}-V_{r\wedge\tau_n}\|^2_\rrm{HS}\big]\Big)\rrm ds,
		\end{equation}
		and 
		\begin{equation}
			\bb E\big[\sup_{0\le s\le t}\|U_{s\wedge\tau_n}-V_{s\wedge\tau_n}\|^2_\rrm{HS}\big]\le C\int_{0}^{t}\Big(\bb E\big[\sup_{0\le r\le s}|Y_{r\wedge\tau_n}-Z_{r\wedge\tau_n}|^2\big]+\bb E\big[\sup_{0\le r\le s}\|U_{r\wedge\tau_n}-V_{r\wedge\tau_n}\|^2_\rrm{HS}\big]\Big)\rrm ds.
		\end{equation}
        By combining the two inequalities and applying Gr\"owall's inequality, we obtain that
		\begin{equation}
			\bb E\big[\sup_{0\le s\le t}|Y_{s\wedge\tau_n}-Z_{s\wedge\tau_n}|^2\big]+\bb E\big[\sup_{0\le s\le t}\|U_{s\wedge\tau_n}-V_{s\wedge\tau_n}\|^2_\rrm{HS}\big]=0,
		\end{equation}
		which implies $Y_{s\wedge\tau_n}=Z_{s\wedge\tau_n}$ for all $s\in[0,T]$ almost surely with respect to $\bb P$. Letting $n\to\infty$, we obtain $Y_{s}=Z_{s}$ for all $s\in[0,T]$ almost surely. 
		
		To show $U_t=V_t$, note that
		$$\bb E\big[\|U_{s\wedge\tau_n}-V_{s\wedge\tau_n}\|^2_\rrm{HS}\big]=0,\, \text{for all } s\in[0,T].$$ 

        The dominated convergence theorem yields that 
		\begin{equation}
				\|U_{s}-V_{s}\|^2_\rrm{HS}=\bb E\big[\|U_{s}-V_{s}\|^2_\rrm{HS}\big]=\bb E\big[\lim_{n\to\infty}\|U_{s\wedge\tau_n}-V_{s\wedge\tau_n}\|^2_\rrm{HS}\big]=\lim_{n\to\infty}\bb E\big[\|U_{s\wedge\tau_n}-V_{s\wedge\tau_n}\|^2_\rrm{HS}\big]=0.
		\end{equation}
		Thus, $U_t=V_t$ for all $t\in[0,T]$. 
        
       Consequently, for any $T>0$, the SDE~\eqref{eq.sec4} has a unique strong DO solution on $(0,T]$.
	\end{proof}	

    In Part IV of the proof of Proposition~\ref{wellposedness}, for the DO solution  the covariance matrix $C_{Y_t}$ is invertible for all $t>0$.

    Notice that in the DLRA~\eqref{eqn:intro_xr}, the factorisation \((U,Y)\) is not unique.  
Since \(X_t^{(r)} = U_t^\top Y_t\), for any \(Q \in O(r)\) we may consider \((\widetilde U_t,\widetilde Y_t) := (Q U_t, Q Y_t)\), which also satisfies \(X_t^{(r)} = \widetilde U_t^\top \widetilde Y_t\).  
The following theorem shows that \((U_t,Y_t)\) and \((\widetilde U_t,\widetilde Y_t)\) result in the same system~\eqref{eq:DO}.  
See also \cite[Proposition 2.10]{DLRASDE} for a similar result.

\begin{theorem}\label{theo:rotation}
Suppose Assumptions~(H1)--(H5) hold, and let $(U_t,Y_t)_{t\ge0}$ be the solution of Eqn.~\eqref{eq:DO} with initial condition $U_0 \in V_r(\bb R^d)$ and $Y_0 \in L^2(\Omega,\bb R^r)$. Assume further that there exists $\sigma_{Y_0} > 0$ such that $\bb E[Y_0 Y_0^\top] \succeq \sigma_{Y_0} I_r$.
For any $Q\in O(r)$, define $\al R_Q(u,y) := (Qu, Qy)$. Then, for any $t>0$ we have
\begin{equation}
    \scr L_{(U_t^{\al R_Q(U_0,Y_0)},Y_t^{\al R_Q(U_0,Y_0)})}=\scr L_{(QU_t^{(U_0,Y_0)},QY_t^{(U_0,Y_0)})},
\end{equation}
where $(U_0^{(U_0,Y_0)},Y_0^{(U_0,Y_0)})=(U_0,Y_0)$ and $(U_0^{\al R_Q(U_0,Y_0)},Y_0^{\al R_Q(U_0,Y_0)})=(QU_0,QY_0)$.
\end{theorem} 

\begin{proof}
Let $(U_{t},Y_{t})$ be the solution of Eqn.~\eqref{eq:DO} with initial
condition $(U_0,Y_0)$. Define 
\[
\widetilde{U}_{t}:=QU_{t},\qquad\widetilde{Y}_{t}:=QY_{t}.
\]
Then 
\[
\widetilde{U}_{t}^{\top}\widetilde{Y}_{t}=U_{t}^{\top}Q^{\top}QY_{t}=U_{t}^{\top}Y_{t}\quad\text{ and }\quad P_{\widetilde{U}_{t}}=\widetilde{U}_{t}^{\top}\widetilde{U}_{t}=U_{t}^{\top}Q^{\top}QU_{t}=P_{U_{t}}.
\]

We first verify that the rotated process $\widetilde{Y}_{t}$ satisfies the same equation as $Y_t$. Since $Q$ is constant in $t$, 
\[
\rrm d\widetilde{Y}_{t}=Q\,\rrm dY_{t}.
\]
Using Eqn.~\eqref{eq:DO} for $Y_{t}$, we get 
\[
\begin{aligned}\rrm d\widetilde{Y}_{t} & =QU_{t}b(U_{t}^{\top}Y_{t})\,\rrm dt+QU_{t}\sigma(U_{t}^{\top}Y_{t})\,\rrm dB_{t}\\
 & =\widetilde{U}_{t}b(\widetilde{U}_{t}^{\top}\widetilde{Y}_{t})\,\rrm dt+\widetilde{U}_{t}\sigma(\widetilde{U}_{t}^{\top}\widetilde{Y}_{t})\,\rrm dB_{t}.
\end{aligned}
\]
Thus the rotated process satisfies the same $Y$-equation.

We now verify that the rotated process $\widetilde{U}_{t}$ satisfies the same equation as $U_t$. Since 
\[
C_{\widetilde Y_t}=\bb E[\widetilde Y_t\widetilde Y_t^\top]=Q\bb E[ Y_tY_t^\top]Q^{\top}=QC_{Y_t}Q^{\top}.
\]
Hence $QC_{Y_t}=C_{\widetilde Y_t}Q$.

Moreover, we have 
\[
\begin{aligned}\bb E[\widetilde Y_tb(t,\widetilde U_t^\top \widetilde Y_t)^\top]=Q\bb E[Y_tb(t,U_t^\top Y_t)^\top].
\end{aligned}
\]
Therefore, 
\[
\begin{aligned}C_{\widetilde Y_t}\,\rrm d\widetilde{U}_{t}  =C_{\widetilde Y_t}Q\,\rrm dU_{t}=QC_{Y_t}\,\rrm dU_{t}
 =Q\bb E[Y_tb(t,U_t^\top Y_t)^\top](I_{d}-P_{{U}_{t}})\,\rrm dt
  =\bb E[\widetilde Y_tb(t,\widetilde U_t^\top \widetilde Y_t)^\top](I_{d}-P_{\widetilde{U}_{t}})\,\rrm dt.
\end{aligned}
\]
Thus $(\widetilde{U}_{t},\widetilde{Y}_{t})$ solves the frozen system
with initial condition $(QU_0,QY_0)$.

By Theorem~\ref{wellposedness}, for every initial condition satisfying the requirement that there exists a constant $\widetilde\sigma>0$ such that $\bb E[Y_0Y_0^\top]\succeq\widetilde\sigma I_r$, Eqn.~\eqref{eq:DO} admits a pathwise unique strong solution whose law is uniquely determined by the initial condition. Therefore, 
\[
\scr L_{(U_t^{QU_0,QY_0},Y_t^{QU_0,QY_0})}=\scr L_{(QU_t^{U_0,Y_0},QY_t^{U_0,Y_0})}.
\]
This proves the claim. 

\end{proof}
	
Recall that $P_{Y_t}=\bb E[\cdot Y_t^\top]C_{Y_t}^{-1}Y_t$. We now state a key property of the operator $P_{Y_t}$ for later use.
    \begin{lemma}\label{le.dim}
        The operator $P_{Y_t}$ is defined from $L^2(\Omega,\bb R^d)$ to itself. Moreover, the norm $$\|P_{Y_t}\|_{L^2(\Omega,\bb R^d)\to L^2(\Omega,\bb R^d)}\le 1.$$
    \end{lemma}
    \begin{proof}
   Note that for any $u, v \in L^2(\Omega,\bb R^d)$,
\[
\langle u, P_{Y_t} v \rangle_{L^2(\Omega,\bb R^d)} = \langle P_{Y_t} u, v \rangle_{L^2(\Omega,\bb R^d)},
\]
where $P_{Y_t} : L^2(\Omega,\bb R^d) \to L^2(\Omega,\bb R^d)$ is the projection onto $\rrm{span}\{y_{t,1}, y_{t,2}, \dots, y_{t,r}\}$ defined component‑wise; that is, each component is projected using the projection in $L^2(\Omega,\bb R)$ onto the same span.

    We now estimate the operator norm of $P_{Y_t}$.  Since $P_{Y_t}$ is the projection from $L^2(\Omega,\bb R^d)$ to $L^2(\Omega,\bb R^d)$, it follows that $P^2_{Y_t}=P_{Y_t}$ on space $L^2(\Omega,\bb R^d)$. For any $u\in L^2(\Omega,\bb R^d)$ satisfying $\bb E[|u|^2]=1$, we have 
		\begin{equation*}
			\begin{aligned}
				\bb E[|P_{Y_t}u|^2]=\left\langle P_{Y_t}u,P_{Y_t}u\right\rangle_{L^2(\Omega,\bb R^d)}=\left\langle P^2_{Y_t}u,u\right\rangle_{L^2(\Omega,\bb R^d)}\le \big(\bb E[|P_{Y_t}u|^2]\big)^\frac{1}{2}\big(\bb E[|u|^2]\big)^\frac{1}{2}.
			\end{aligned}
		\end{equation*}
		Taking the supremum over such $u$, it follows that
		$$\|P_{Y_t}\|_{L^2(\Omega,\bb R^d)\to L^2(\Omega,\bb R^d)}=\sup_{u\in L^2(\Omega,\bb R^d) \text{ and }\bb E[|u|^2]=1}\big(\bb E[|P_{Y_t}u|^2]\big)^\frac{1}{2}\le 1.$$
    \end{proof}
	Next, we establish the consistency of the DO solution in the following sense: when $r=d$, the DO solution coincides with the original solution.
    \begin{proposition}\label{PropXandXDO}
    Suppose Assumptions (H1)--(H5) hold. Let $T>0$ and suppose $r=d$. If $X^{(r)}_t$ and $X_t$ with the same initial data $\xi\in L^2(\Omega,\bb R^d)$ both exist on $[0,T]$. Then, it follows that
		$$\bb E[\sup_{t\in[0,T]}|X^{(r)}_t-X_t|^2]=0.$$
	\end{proposition}
	\begin{proof}
		When $r=d$, we have span$\{u_{t,1},u_{t,2},\cdots,u_{t,d}\}=\bb R^d$ and that $u_{t,i}$ are orthogonal basis in $\bb R^d$. Then $ P_{U_t}=U_t^\top U_t=I_{d}=U_tU_t^\top.$ Substituting it into Eqn.~\eqref{eqn:DLRA}, it yields that
		$$\rrm dX_t^{(r)}=b(t,X_t^{(r)})\,\rrm dt+\sigma(t,X^{(r)}_{t})\,\rrm d B_t,$$
        which is exactly the original SDE. 
       The result follows from the uniqueness of solution to Eqn.~\eqref{eq.sec4}.
	\end{proof}
	
	\section{Uniform-in-Time Estimation }\label{sec:uniform}
	In this section, we mainly discuss the Wasserstein distance between the solution of Eqn.~\eqref{eq.sec4} and its DLRA Eqn.~\eqref{eqn:DLRA} under the global disspative condition. 
	Denote by $\mu_t$ and $\nu_t$ the laws of $X_t$ and $X^{(r)}_t$, respectively, where $\scr L_{X_0}=\mu$ and $\scr L_{X^{(r)}_0}=\nu$. 
    
	\begin{assumption}
	    \begin{itemize}
	        \item[(H6)]
            There exists $L_{b,3}\in\bb R$ such that for all $t\in[0,\infty)$ and $x,y\in\bb R^d$, one has 
			\begin{align*}
				\left\langle x-y,b(t,x)-b(t,y)\right\rangle\le L_{b,3}|x-y|^2.
			\end{align*}
	    \end{itemize}
	\end{assumption}
    To obtain a uniform-in-time estimate, we require the following relation on the constants.
    \begin{assumption}\label{AS:relation1}
       The constants $L_{\sigma,1}$ and $L_{b,3}$ (appearing in (H3) and (H6), respectively) satisfy $2L_{\sigma,1} + 2L_{b,3} < -1$.
    \end{assumption}
		
	\begin{remark}
		Suppose  Assumptions~(H1)--(H4) and~(H6) hold. Then, for all $t\in[0,\infty)$ and $x\in \bb R^d$, we have 
		\begin{equation}\label{eqn:prop}
			\begin{aligned}
				2\left\langle x,b(t,x)\right\rangle+\|\sigma(t,x)\|^2_\rrm{HS}\le (2L_{b,3}+1+2L_{\sigma,1})|x|^2+(L_{b,2}^2+2L_{\sigma,2}).
			\end{aligned}
		\end{equation}
        {In this section we work on the uniform-in-time estimate of $X^{(r)}_t$, which requires the coefficient of $|x|^2$ on the right‑hand side of \eqref{eqn:prop} to be negative. This is precisely the reason for imposing Assumption~\ref{AS:relation1}. If that coefficient were positive, only a finite‑time estimate could be obtained.}
	\end{remark}
    The following result provides uniform-in-time bounds for the DLRA method.
	\begin{theorem}\label{bounded}
    Suppose  Assumptions~(H1)--(H6) hold. Moreover, suppose Assumption~\ref{AS:relation1} holds. 
    Then, there exists a constant $C(b,\sigma)>0$ such that for any $r\in\{1,2,\cdots,d\}$,
		\begin{equation}\label{eq,DObound}
			\sup_{t\in[0,\infty)}\bb E[|X^{(r)}_t|^2]\le C(b,\sigma)\big(1+\bb E[|X^{(r)}_0|^2]\big).
		\end{equation}
		In particular, for $r=d$ we have 
		$
			\sup_{t\in[0,\infty)}\bb E[|X_t|^2]\le C(b,\sigma)\big(1+\bb E[|X_0|^2]\big)$.
	\end{theorem}
	\begin{proof}
		Since $X^{(r)}_t$ is a stochastic process on the $\rrm{span}\{u_{t,1},u_{t,2},\cdots, u_{t,r}\}$ and $P_{U_t}$ is the orthogonal matrix, for any function $f: [0,\infty) \times \bb R^d \to \bb R^d$, we have  
		\begin{equation}\label{eqn:projU}
			\begin{aligned}
				\big\langle X^{(r)}_t,P_{U_t}f(t,X^{(r)}_t) \big\rangle=\big\langle P_{U_t}X^{(r)}_t,f(t,X^{(r)}_t) \big\rangle=\big\langle X^{(r)}_t,f(t,X^{(r)}_t) \big\rangle.
			\end{aligned}
		\end{equation}
		Let $\rho:=2L_{b,3}+2L_{\sigma,1}+1<0$. 
        Applying It\^o's formula together with Eqns.~\eqref{eqn:prop} and \eqref{eqn:projU} to the DO approximation \eqref{eqn:DLRA} yields
		\begin{align*}
			\rrm d\rrm e^{-\rho t}|X^{(r)}_t|^2&\le -\rho\rrm e^{-\rho t}|X^{(r)}_t|^2\,\rrm dt+2\rrm e^{-\rho t}\big\langle X^{(r)}_t,\big((I_{d}-P_{U_t})\big(P_{Y_t}b(t,X^{(r)}_t)\big)+P_{U_t}b(t,X^{(r)}_t)\big) \big\rangle\,\rrm dt\\
			&\qquad+\rrm e^{-\rho t}\|P_{U_t}\sigma(t,X^{(r)}_t)\|^2_\rrm{HS} \,\rrm dt+\rrm dM_t\\&\le \rrm e^{-\rho t}(L_{b,2}^2+2L_{\sigma,2})\,\rrm dt+\rrm dM_t,
		\end{align*}
		where $M_t=\int_0^t\rrm e^{-\rho s}\big\langle X_s^{(r)},P_{U_s}\sigma(s,X_s^{(r)})\,\rrm d B_s
        \big\rangle$. For any $N\in\bb N$, define $\varrho_N=\inf\{t\ge0:|X^{(r)}_t|>N\}$. Since $|X^{(r)}_{t\wedge\varrho_N}|<N$ bounds the integrand, $\bb E[M_{t\wedge\varrho_N}]=0$; therefore, we have 
		\begin{align*}
			\bb E\big[\rrm e^{-\rho (t\wedge\varrho_N)}|X^{(r)}_{t\wedge\varrho_N}|^2\big]\le \bb E[|X_0^{(r)}|^2]+\frac{L_{b,2}^2+2L_{\sigma,2}}{\rho}(\rrm e^{-\rho t}-1).
		\end{align*}
        By Fatou's lemma, we have 
        \begin{equation}
            \rrm e^{-\rho t}\bb E\big[|X^{(r)}_{t}|^2\big]=\liminf_{N\to\infty}\bb E\big[\rrm e^{-\rho (t\wedge\varrho_N)}|X^{(r)}_{t\wedge\varrho_N}|^2\big]\le \bb E[|X_0^{(r)}|^2]+\frac{L_{b,2}^2+2L_{\sigma,2}}{\rho}(\rrm e^{-\rho t}-1).
        \end{equation}
		Thus 
		\begin{equation}\label{eq:bound}
			\bb E[|X^{(r)}_t|^2]\le \Big(\frac{L_{b,2}^2+2L_{\sigma,2}}{\rho}(\rrm e^{-\rho t}-1)+\bb E[|X^{(r)}_0|^2]\Big) \rrm e^{\rho t}\le C(b,\sigma)(1+\bb E[|X^{(r)}_0|^2]).
		\end{equation}
        When $r=d$, Proposition~\ref{PropXandXDO} together with \eqref{eq:bound} yields
		 \begin{equation*}
		 	\begin{aligned}
				\bb E|X_t|^2\le 2\bb E|X_t-X^{(d)}_t|^2+2\bb E|X^{(d)}_t|^2&\le 0+ 2C(b,\sigma)(1+\bb E|X_0|^2).
		 	\end{aligned}			
		 \end{equation*}
	\end{proof}
	We now state the main theorem of this section. 
	\begin{theorem}\label{maintheorem}
		Suppose Assumptions~(H1)--(H6) hold. Moreover,
		suppose Assumption~\ref{AS:relation1} holds. Then, there exists a constant $C(\mu,\nu,b,\sigma)>0$ such that, for all $t\in[0,\infty)$,
		\begin{equation}\label{eq.theorem}
			\bb W_2^2(\mu_t,\nu_t)\le \rrm e^{\delta t}\bb W_2^2(\mu,\nu)+\frac{C(\mu,\nu,b,\sigma)}{|\delta|}(d-r).
		\end{equation}
	\end{theorem}
	\begin{proof}
		When $r<d$, consider the following system of SDE,
		\begin{align*}
			\rrm d X_t&=b(t,X_t)\,\rrm dt+\sigma(t,X_t)\,\rrm d B_t;\\
			\rrm d X^{(r)}_t&=\big((I_{d}-P_{U_t})\big(P_{Y_t}b(t,X^{(r)}_t)\big)+P_{U_t}b(t,X^{(r)}_t)\big)\,\rrm dt+P_{U_t}\sigma(t,X^{(r)}_t)\,\rrm d B_t.
		\end{align*}
		Let $I_{t,1}$ and $I_{t,2}$ be defined by
		\begin{align}
			I_{t,1}&=2\big\langle X_t-X^{(r)}_t,b(t,X_t)-b(t,X^{(r)}_t)\big\rangle+2\big\|\sigma(t,X_t)-\sigma(t,X^{(r)}_t)\big\|^2_\rrm{HS},\\
			I_{t,2}&=2\big\langle X_t-X^{(r)}_t,b(t,X^{(r)}_t)-\big((I_{d}-P_{U_t})\big(P_{Y_t}b(t,X^{(r)}_t)\big)+P_{U_t}b(t,X^{(r)}_t)\big)\big\rangle\nonumber\\&\quad+2\big\|\sigma(t,X^{(r)}_t)-P_{U_t}\sigma(t,X^{(r)}_t)\big\|^2_\rrm{HS}.
		\end{align}
        By applying It\^o's formula, we obtain
		\begin{align*}
			\rrm d \rrm e^{-\delta t}|X_t-X^{(r)}_t|^2&=2\rrm e^{-\delta t}\big\langle X_t-X^{(r)}_t,b(t,X_t)-\big((I_{d}-P_{U_t})\big(P_{Y_t}b(t,X^{(r)}_t)\big)+P_{U_t}b(t,X^{(r)}_t)\big)\big\rangle \,\rrm dt\\&\quad-\delta\rrm e^{-\delta t}|X_t-X^{(r)}_t|^2\,\rrm dt+\|\sigma(t,X_t)-P_{U_t}\sigma(t,X^{(r)}_t)\|^2_\rrm{HS}\,\rrm dt+\rrm dM_t\\&=-\delta\rrm e^{-\delta t}|X_t-X^{(r)}_t|^2\,\rrm dt+\rrm e^{-\delta t}I_{t,1}\,\rrm dt+\rrm e^{-\delta t}I_{t,2}\,\rrm dt+\rrm dM_t,
		\end{align*}
		where $M_t=2\int_0^t\rrm e^{-\delta s}\big\langle X_s-X_s^{(r)},\sigma(s,X_s)-P_{U_s}\sigma(s,X_s^{(r)})\,\rrm d B_s\big\rangle$. For any $N\in\bb N$, let $\bar\tau_N=\inf\{t\ge0:|X_t|\vee|X_t^{(r)}|>N\}$. Since $|X_{t\wedge\bar\tau_N}|,\,|X^{(r)}_{t\wedge\bar\tau_N}|\le N$ and (H4) holds, $\bb E[M_{t\wedge\bar\tau_N}]=0$; therefore, we have
		\begin{equation}\label{eqn:x-xr}
        \begin{aligned}
          \bb E[\rrm e^{-\delta (t\wedge\bar\tau_N)}|X_{t\wedge\bar\tau_N}-X^{(r)}_{t\wedge\bar\tau_N}|^2]&\le\bb E[|X_{0}-X^{(r)}_{0}|^2]+\int_0^t\big(\rrm e^{-\delta s}\bb E[I_{s,1}]-\delta\rrm e^{-\delta s}\bb E[|X_s-X^{(r)}_s|^2]\big)\,\rrm ds\\&\quad+\int_0^t\rrm e^{-\delta s}\bb E[I_{s,2}]\,\rrm ds.  
        \end{aligned}
		\end{equation}
        Next we will estimate $\bb EI_{s,1},\bb EI_{s,2}$. Based on (H3) and (H6), we have that
		\begin{equation}\label{eq.I1}
			\bb E[I_{s,1}]\le 2\big(L_{b,3}+L_{\sigma,1}\big)\bb E[|X_s-X^{(r)}_s|^2].
		\end{equation}
		For the term $I_{s,2}$, we also have 
		\begin{align*}
			\bb E[I_{s,2}]&=2\bb E\big\langle X_s-X^{(r)}_s,b(s,X^{(r)}_s)-\big((I_{d}-P_{U_s})\big(P_{Y_s}b(t,X^{(r)}_s)\big)+P_{U_s}b(s,X^{(r)}_s)\big)\big\rangle\\&\quad+\bb E[\|\sigma(s,X^{(r)}_s)-P_{U_s}\sigma(s,X^{(r)}_s)\|^2_\rrm{HS}]\\&\le 
			2\bb E\big\langle\big(I_{d}-P_{U_s}\big) \big(X_s-X^{(r)}_s\big),b(s,X^{(r)}_s)-P_{Y_s}b(s,X^{(r)}_s)\big\rangle\\&\quad+2\|I_{d}-P_{U_s}\|_\rrm{HS}^2\bb E[\|\sigma(s,X^{(r)}_s)\|^2_\rrm{HS}].
		\end{align*}
		By applying Young’s inequality and using the bound in Eqn.~\eqref{eq,DObound}, together with the estimate from (H2), Lemma~\ref{le.dim} and Theorem~\ref{bounded}, we obtain
		\begin{align*}
			&2\bb E\big\langle\big(I_{d}-P_{U_s}\big) \big(X_s-X^{(r)}_s\big),b(t,X^{(r)}_s)-P_{Y_s}b(t,X^{(r)}_s)\big\rangle,\\&\le 2\bb E\big\langle\big(I_{d}-P_{U_s}\big)X_s,b(s,X^{(r)}_s)-P_{Y_s}b(t,X^{(r)}_s)\big\rangle,\\&\le
			2\bb E\big\langle\big(I_{d}-P_{U_s}\big),\big(b(s,X^{(r)}_s)-P_{Y_s}b(s,X^{(r)}_s)\big)X_s^\top\big\rangle_\rrm{HS},
			\\&\le 2\|I_d-P_{U_s}\|_\rrm{HS}\bb E\big[|\big(b(s,X^{(r)}_s)-P_{Y_s}b(s,X^{(r)}_s)\big)||X_s|\big],\\&\le \|I_d-P_{U_s}\|_\rrm{HS}\big(\bb E[|b(s,X^{(r)}_s)-P_{Y_s}b(s,X^{(r)}_s)|^2]+\bb E[|X_s|^2]\big),\\&\le 
            C(\mu,\nu,b,\sigma)\sqrt{d-r}.
		\end{align*}
		where $C(\mu,\nu,b,\sigma)$ is a constant independent of time $s$.
        
		  According to the property of projection $P_{U_s}$, (H4) and Theorem~\ref{bounded}, we have 
		  \begin{align*}
		 	\|I_{d}-P_{U_s}\|_\rrm{HS}^2\bb E[\|\sigma(s,X^{(r)}_s)\|^2_\rrm{HS}]\le (d-r)L_{\sigma,2}\big(1+\bb E[|X^{(r)}_s|^2]\big)
		 	\le C(\mu,\nu,b,\sigma)(d-r).
		 \end{align*}		
	Then we obtain that 
		\begin{equation}\label{eq.I2}
			\begin{aligned}
				\bb E[I_{t,2}]&\le C(\mu,\nu,b,\sigma)(d-r).
			\end{aligned}
		\end{equation}
		Combining Eqns.~\eqref{eqn:x-xr}, \eqref{eq.I1} and \eqref{eq.I2} yields
		\begin{align*}
			\bb E[\rrm e^{-\delta (t\wedge\bar\tau_N)}|X_{t\wedge\bar\tau_N}-X^{(r)}_{t\wedge\bar\tau_N}|^2]&\le\bb E[|X_{0}-X^{(r)}_{0}|^2]+\frac{C(\mu,\nu,b,\sigma)}{|\delta|}(1-\rrm e^{-\delta t})(d-r).
		\end{align*}
		Finally, from the definition of the Wasserstein distance, it follows that
		\begin{align*}
			\bb W_2^2(\mu_t,\nu_t)\le\bb E[|X_t-X^{(r)}_t|^2]\le \rrm e^{\delta t}\bb E[|X_0-X_0^{(r)}|^2]+\frac{ C(\mu,\nu,b,\sigma)}{|\delta|}(d-r).
		\end{align*}      
		Taking infimum with respect to $X_0,X_0^{(r)}$ satisfying $\scr L_{X_0}=\mu,\ \scr L_{X^{(r)}_0}=\nu$, we have 
		\begin{equation}\label{eq.d<r}
			\bb W_2^2(\mu_t,\nu_t)\le \rrm e^{\delta t}\bb W_2^2(\mu,\nu)+\frac{ C(\mu,\nu,b,\sigma)}{|\delta|}(d-r).
		\end{equation}
        When $d=r$, Proposition~\ref{PropXandXDO} reduces the analysis to the following system of SDEs
        \begin{equation}\label{eqn:sym}
            \begin{aligned}
			\rrm d X_t&=b(t,X_t)\,\rrm dt+\sigma(t,X_t)\,\rrm d B_t;\\
			\rrm d X^{(r)}_t&=b(t,X^{(r)}_t)\,\rrm dt+\sigma(t,X^{(r)}_t)\,\rrm d B_t.
		\end{aligned}
        \end{equation}
        Applying the same argument as before, we obtain 
        \begin{equation}\label{eq.d=r}
            \bb W_2^2(\mu_t,\nu_t)\le \rrm e^{\delta t}\bb W_2^2(\mu,\nu).
        \end{equation}
        Combining Eqn.~\eqref{eq.d<r} and Eqn.~\eqref{eq.d=r} yields the desired result.
	\end{proof}	

	\subsection{Sharpness of Theorem~\ref{maintheorem}}\label{ssec:sharpness}
	In Theorem~\ref{maintheorem}, the upper bound in Eqn.~\eqref{eq.theorem} has a constant term that depend on $(d-r)$, and thus if $d> r$ then this bound is positive. 
	It turns out that this term is not improvable in general, as the following example shows.
	
	\begin{example}\label{exam:sharp}
		Consider the following SDE in $\bb R^d$,
		\begin{equation}\label{ex.optimal}
			\rrm dX_t=-X_t\,\rrm dt+\rrm d W_t,
		\end{equation}
		where the drift and diffusion coefficients are given by $b(t,x)=-x$ and $\sigma(t,x)=1$, respectively. Here $(W_t)_{t\ge0}$ denotes a $d$-dimensional standard Brownian motion. Let $\scr L_{X_0}=\al N(0,I_d)$ and $\scr L_{X_0^{(r)}}=\al N(0,A)$, where $$A=\rrm {diag}(\underbrace{1,1,\cdots,1}_r,0,\cdots,0).$$ Then, for all $t\ge0$, it holds that $$\bb W_2^2\big(\scr L_{X^{(r)}_t},\scr L_{X_t}\big)= \rrm e^{-2t}\bb W^2_2\big(\scr L_{X_0},\scr L_{X_0^{(r)}}\big)+\frac{1-\rrm e^{-2t}}{2}(d-r).$$
	\end{example}
	It is well known that the solution of Eqn.~\eqref{exam:sharp} is 
	\begin{equation}
		X_t=\rrm e^{-t}X_0+\int_{0}^{t}\rrm e^{-(t-s)}\,\rrm dW_s,
	\end{equation}
	where $X_0$ denotes the initial condition. Consequently, the law of $X_t$ is given by
	\begin{equation}\label{eq.disofSDE}
		\scr L_{X_t}=\al N\big(\rrm e^{-t}\bb E[X_0],C_t\big),
	\end{equation}
	where $$C_t=\rrm e^{-2t}C_{X_0}+\frac{1-\rrm e^{-2t}}{2}I_{d},$$ and $C_{X_0}$ is the covariance matrix of $X_0$.
	
From Eqn.~\eqref{eq:DO}, the strong DO solution of Eqn.~\eqref{ex.optimal} satisfies
	\begin{equation}\label{ex.DO}
		\begin{aligned}
			\rrm d U_t&=0\,\rrm dt,\\
			\rrm d Y_t&=-Y_t\,\rrm dt+U_t\,\rrm dW_t,
		\end{aligned}
	\end{equation}
	with initial condition $(U_0,Y_0)$. Assume that there exists a constant $\sigma_{Y_0}>0$ such that $\bb E[Y_0Y^\top_0]\succeq\sigma_{Y_0}I_{r}$. Then, by \cite[Proposition 4.5]{DLRASDE}, the covariance matrix $C_{Y_t}$ is invertible for all $t\ge0$. The system \eqref{ex.DO} admits the unique solution
	\begin{equation}
		\begin{aligned}
			U_t&=U_0,\\
			Y_t&=\rrm e^{-t}Y_0+\int_{0}^{t}\rrm e^{-(t-s)}U_0\,\rrm dW_s.
		\end{aligned}
	\end{equation}
	The $r$-dimensional DO approximation of Eqn.~\eqref{ex.optimal} is 
	\begin{equation}
		X_t^{(r)}=U^\top_tY_t=\rrm e^{-t}X_0^{(r)}+\int_{0}^{t}\rrm e^{-(t-s)}U^\top_0U_0\,\rrm dW_s.
	\end{equation}
		Similarly, the distribution of $X_t^{(r)}$ is 
	\begin{equation}
		\scr L_{X^{(r)}_t}=\al N\big(\rrm e^{-t}\bb E[X^{(r)}_0],C^{(r)}_t\big),
	\end{equation}
	where $$C^{(r)}_t=\rrm e^{-2t}C_{X^{(r)}_0}+\frac{1-\rrm e^{-2t}}{2}U_0^\top U_0,$$ and  $C_{X^{(r)}_0}$ denotes the covariance matrix of $X^{(r)}_0$. 
	
	Next, assume that $X_0$ is normally distributed with mean $0$ and covariance matrix $I_{d}$. Likewise, assume that $Y_0$ is normally distributed with mean $0$ and covariance matrix $I_{r}$.
	By~\cite[Proposition 7]{normaldis}, the squared $2$‑Wasserstein distance between the laws of $X_t$ and $X_t^{(r)}$ is given by
	\begin{align*}
		\bb W^2_2\big(\scr L_{X^{(r)}_t},\scr L_{X_t}\big)&=\|\rrm e^{-t}\bb E[X_0]-\rrm e^{-t}\bb E[X^{(r)}_0]\|^2+\rrm{Tr}(C_t)+\rrm {Tr}(C^{(r)}_t)-2\rrm{Tr}\big[\big(C_t^{1/2}C^{(r)}_tC_t^{1/2}\big)^{1/2}\big].
	\end{align*}
	Note that 
	\begin{align*}
		C^{1/2}_t=\sqrt{\frac{1+\rrm e^{-2t}}{2}}I_{d}.
	\end{align*}
	Since $$C_{X_0^{(r)}}=\bb E[X_0^{(r)}X_0^{(r)\top}]=U_0^\top\bb E[Y_0Y_0^{\top}]U_0,$$
	the covariance matrix $C_t^{(r)}$ admits the spectral representation
	\begin{align*}
		C^{(r)}_t=\frac{1+\rrm e^{-2t}}{2}Q\rrm {diag}(\underbrace{1,1,\cdots,1}_r,0,\cdots,0)Q^\top,
	\end{align*}
	for some orthogonal matrix $Q\in\bb R^{d\times d}$.
	
	Therefore, 
	\begin{align*}
		\rrm{Tr}\big[\big(C_t^{1/2}C^{(r)}_tC_t^{1/2}\big)^{1/2}\big]=\frac{1+\rrm e^{-2t}}{2} r.
	\end{align*}
	We also have 
	\begin{align*}
		\rrm{Tr}\big[\big(C_{X_0}^{1/2}C_{X^{(r)}_0}C_{X_0}^{1/2}\big)^{1/2}\big]= r.
	\end{align*}
	Combining these identities yields
	\begin{align*}
		\bb W^2_2\big(\scr L_{X^{(r)}_t},\scr L_{X_t}\big)&=\|\rrm e^{-t}\bb E[X_0]-\rrm e^{-t}\bb E[X^{(r)}_0]\|^2+\rrm{Tr}(C_t)+\rrm {Tr}(C^{(r)}_t)-2\rrm{Tr}\big[\big(C_t^{1/2}C^{(r)}_tC_t^{1/2}\big)^{1/2}\big]\\&
		=\rrm e^{-2t}\|\bb E[X_0]-\bb E[X^{(r)}_0]\|^2+\rrm e^{-2t}\rrm{Tr}\big(C_{X_0}\big)+\rrm e^{-2t}\rrm{Tr}\big(C_{X^{(r)}_0}\big)-2\rrm e^{-2t}\rrm{Tr}\big[\big(C_{X_0}^{1/2}C_{X^{(r)}_0}C_{X_0}^{1/2}\big)^{1/2}\big]\\&\quad + \frac{1-\rrm e^{-2t}}{2}d+\frac{1-\rrm e^{-2t}}{2} r-(1+\rrm e^{-2t})r+2\rrm e^{-2t}r\\&=\rrm e^{-2t}\bb W^2_2\big(\scr L_{X_0},\scr L_{X_0^{(r)}}\big)+\frac{1-\rrm e^{-2t}}{2}(d-r).
	\end{align*}
	Thus, this construction provides an explicit example illustrating that the estimate in Theorem ~\ref{maintheorem} is optimal. By taking $b=-I_d$ and $\sigma=1$, we obtain $L_{b,3}=-1$, $L_{\sigma,2}=0$, which satisfy the assumptions of Theorem~\ref{maintheorem}. In addition, $1-\rrm e^{-2t}\le 1$, with the constant $C=1$ in this example.

	\section{Invariant probability measure}\label{sec:invariant}		
	In this section, we study the invariant probability measure of the pair $(U,Y)$. 
To ensure the existence of an invariant measure, we restrict ourselves to time‑independent coefficients in SDE~\eqref{eq.sec4} and impose suitable assumptions. 
In the classical DO decomposition, the initial value $U_0$ is usually taken as a deterministic matrix in $V_r(\bb R^d)$. However, to study invariant measures we need to consider random initial conditions; this is motivated by Example~\ref{exam:notdirac} presented later in this section.

We define the space $V_r(\bb R^d)$ by
$$
V_r(\bb R^d) = \{ A \in \bb R^{r \times d} : AA^\top = I_r \}.
$$
Note that, by \cite[Chapter 1]{ash2000probability}, it holds that
    \begin{equation*}
        \scr B(V_r(\bb R^d))=\{B \cap V_r(\bb R^d): B\in\scr B(\bb R^{r\times d})\}.
    \end{equation*}   
Recall that $\scr P(V_r(\bb R^d))$ is the set of all probability measures on the measurable space $(V_r(\bb R^d),\scr B(V_r(\bb R^d)))$.
It is easy to see that for any $\varpi\in \scr P(V_r(\bb R^d))$ and any $p>0$, we have 
    \begin{equation}\label{eq:Vr}
        \int_{V_r(\bb R^d)}\|x\|_\rrm{HS}^p\,\varpi(\rrm dx)=r^\frac{p}{2}\varpi(V_r(\bb R^d))=r^\frac{p}{2}<\infty.
    \end{equation}
    
Since the DO solution $(U,Y)$ given by Eqn.~\eqref{eq:DO} takes values in $V_r(\bb R^d) \times \bb R^r$, we consider the product space equipped with the norm
\begin{equation}\label{eqn:norm}
    \|(v,y)\| := \bigl( \|v\|_{\rrm{HS}}^2 + |y|^2 \bigr)^{1/2}, \qquad (v,y) \in V_r(\bb R^d)\times \bb R^r.
\end{equation}
When the system~\eqref{eq:DO} is well‑posed, we denote by $\al P^*_t \pi := \scr L_{(U_t,Y_t)}$ the distribution of the solution with initial distribution $\pi \in \scr P(V_r(\bb R^d) \times \bb R^r)$. Then, for every bounded measurable function
$g:V_r(\mathbb R^d)\times \mathbb R^r\to \mathbb R$, we have
$$
\int_{V_r(\mathbb R^d)\times \mathbb R^r}
g(u,y)\,\al P_t^*\pi(\rrm du,\rrm dy)
=
\mathbb E\big[g(U_t^{\pi},Y_t^{\pi})\big]$$
where $\scr L_{(U_0,Y_0)}=\pi.$

For fixed $\kappa>0$, define the set $\scr P_2^{\kappa}(\bb R^r)$ by 
    \begin{equation}\label{eqn:P_2k}
          \scr P_2^{\kappa}(\bb R^r)=\Big\{\mu\in \scr P_2(\bb R^r): C_\mu\succeq \kappa I_{r\times r}\Big\},
      \end{equation}
      where $C_\mu=\int_{\bb R^r}xx^\top\mu(\rrm dx)$. 

	An invariant measure for the process $(U,Y)$ is defined as follows.
	\begin{definition}\label{def:inv}
		A probability measure $\pi^*$ on $V_r(\bb R^d)\times \bb R^r$ is called invariant probability measure of a process $(U_t,Y_t)_{t\ge 0}$, if $\al P^*_t \pi^* = \pi^*$ for all $t > 0$. Equivalently, for every bounded measurable functions $f:V_r(\bb R^d)\times \bb R^r\to\bb R$ satisfies 
			$$\pi^*(f)=  \pi^*\big( f(U_t,Y_t)\big) ,\quad \forall\,  t> 0,$$
        where $$\pi^*(f):=\int_{V_r(\bb R^d)\times \bb R^r} f(u,y)\,\pi^*(\rrm du,\rrm dy)$$ and $$\pi^*\big( f(U_t,Y_t)\big):=\bb E[f(U_t,Y_t)],$$ with $\scr L_{(U_0,Y_0)}=\pi^*$.
	\end{definition}  
    In discussing the invariant distribution of $(U,Y)$, one might expect that if $X$ 
		admits an invariant distribution, then the corresponding pair $(U,Y)$ admits a 
		Dirac measure for $U$ and an invariant measure for $Y$. We will show later by example 
		that this is not the case. Intuitively, since the DO decomposition of a DLRA $X^{(r)}$ is unique only up to rotations, this non-uniqueness must be accounted for. 
        
   To obtain an invariant probability measure for Eqn.~\eqref{eq:DO}, we introduce the following assumption.
	\begin{assumption}\label{as.bo4}
		The coefficients $b$ and $\sigma$ are measurable and time-independent. Moreover, the following conditions hold
		\begin{itemize}
			\item[(H7)]  There exists $L_{b,6}\ge0$ such that for all $x,y\in \bb R^d$, one has 
			\begin{align*}
				|b(x)-b(y)|&\le L_{b,6}|x-y|.
			\end{align*}
			\item[(H8)] There exists $L_{b,7}>0$ and $L_{b,8}\ge0$ such that for all $x\in\bb R^d$, one has 
			\begin{align*}
				\left\langle x,b(x) \right\rangle &\le -L_{b,7}|x|^2+L_{b,8}.
			\end{align*}
		\end{itemize}
	\end{assumption}
    \begin{remark}
        Under time‑independent coefficients, (H3) and (H7) imply (H2) and (H4) because global Lipschitz implies linear growth.  Therefore, the assumptions in this section are stronger than those in Section~\ref{sec:uniform}, except for (H6) in Section~\ref{sec:uniform}, because (H6) is stronger than (H8).
    \end{remark}
    \begin{assumption}\label{AS:relation2}
        For some $p> 2$, the constants $L_{\sigma,1}$ and $L_{b,7}$ (appearing in (H3) and (H8), respectively) satisfy $2L_{b,7}-2(p-1)L_{\sigma,1}>1$.
    \end{assumption}
    Next, we give the main result of this section, which ensures the existence of an invariant measure with a non‑degenerate second marginal.
	\begin{theorem}\label{inv}
		Suppose Assumptions  (H3), (H5), (H7) and (H8) hold for SDE~\eqref{eq.sec4}. Moreover, for some $p> 2$, suppose Assumption~\ref{AS:relation2} holds. Then, the strong DO solution~\eqref{eq:DO} to Eqn.~\eqref{eq.sec4} admits an invariant probability measure $\pi^*\in\scr P_2(V_r(\bb R^d)\times\bb R^r)$. Moreover, the second marginal $$\mu^*(A):=\int_{V_r(\bb R^d)}\pi^*(\rrm du, A), \quad\text{for any }A\in \scr B(\bb R^r),$$
        satisfies $\rrm{det}(C_{\mu^*}) \neq 0$, where $C_{\mu^*}=\int_{\bb R^r}xx^\top \, \mu^*(\rrm dx).$
	\end{theorem}
The proof is postponed to the end of this Section. Define $O(r)$ to be the set of all orthogonal matrices in $\bb R^{r\times r}$. 

A probability measure $\eta\in\scr P(E)$ is called $O(r)$-invariant if, for every $Q\in O(r)$ and every $\varphi\in C_{b}(E)$, 
\[
\int_{E}\varphi(Q(x))\,\eta(\rrm dx)=\int_{E}\varphi(x)\,\eta(\rrm dx).
\]
Here the state space $E$ can be taken as $V_r(\bb R^r),\,\bb R^r$ or $V_r(\bb R^r)\times\bb R^r$
In the case $E=V_r(\bb R^r)\times\bb R^r$, the action of  $Q\in O(r)$ on any $(x,y)\in E$ is defined by $Q(x,y)=(Qx,Qy)$.

We now provide another invariant measure for the strong DO solution~\eqref{eq:DO}, which enjoys the $O(r)$-invariance property.
\begin{theorem}\label{cor:inv}
     Given any invariant probability measure $\pi^*$ of Eqn.~\eqref{eq:DO}  whose second marginal $$\mu^*(A):=\int_{V_r(\bb R^d)}\pi^*(\rrm du, A), \quad\text{for any }A\in \scr B(\bb R^r),$$
    satisfies $\rrm{det}(C_{\mu^*}) \neq 0$, there exists another invariant probability measure $\widetilde\pi^*\in\scr P_2(V_r(\bb R^d)\times\bb R^r)$ that is $O(r)$-invariant. Moreover, define $$\widetilde\mu^*(A)=\int_{V_r(\bb R^d)}\widetilde\pi^*(\rrm du, A), \quad\text{for any }A\in \scr B(\bb R^r).$$
    Then we have $\rrm{det}(C_{\widetilde\mu^*}) \neq 0$.
\end{theorem}
\begin{proof}
Since there exists an invariant probability measure $\pi^*\in\scr P_2(V_r(\bb R^d)\times\bb R^r)$ of Eqn.~\eqref{eq:DO}, we average the joint invariant measure $\pi^*$ over $O(r)$. Define $T$ on $C_{b}(V_r(\bb R^d)\times \bb R^r)$ by 
\[
T(f):=\int_{O(r)}\int f(Qu,Qy)\,\pi^*(\rrm du,\rrm dy)\,\rrm dQ,\qquad f\in C_{b}(V_r(\bb R^d)\times \bb R^r).
\]
Since $Q\mapsto\int f(Qu,Qy)\,\pi(\rrm du,\rrm dy)$ is continuous and $O(r)$
is compact, the right-hand side defines a positive linear functional
on $C_{b}(V_r(\bb R^d)\times \bb R^r)$ with value $1$ at $f\equiv1$. Hence, by the Riesz representation
theorem~\cite[Theorem 14.9]{aliprantis2006infinite}, there exists a unique probability measure $\widetilde{\pi}^*$ such that 
$$T(f)=\int f(u,y)\, \widetilde \pi^*(\rrm du,\rrm dy),\quad f\in C_{b}(V_r(\bb R^d)\times \bb R^r).$$
For $f\in C_{b}(V_r(\bb R^d)\times \bb R^r)$, the map 
\[
Q\mapsto\int f(Qu,Qy)\,\pi^*(\rrm du,\rrm dy)
\]
is continuous. Hence the right-hand side defines a positive linear functional on $C_{\mathrm{b}}(V_r(\bb R^d) \times \bb R^r)$, and therefore a probability measure $\widetilde{\pi}^*$.

We claim that $\widetilde{\pi}^*$ is an invariant measure for Eqn.~\eqref{eq:DO}. Indeed, for any $t>0$ and any $f\in C_{\mathrm{b}}(V_r(\bb R^d)\times\bb R^r)$, by Theorem~\ref{theo:rotation}, we have
\[
\begin{aligned}\int_{V_r(\bb R^d)\times\bb R^r} P_tf(u,y)\,\widetilde{\pi}^*(\rrm du,\rrm dy)&=\int_{V_r(\bb R^d)\times\bb R^r} \bb E[f(U_t^{u,y},Y_t^{u,y})]\,\widetilde{\pi}^*(\rrm du,\rrm dy) \\& =\int_{O(r)}\int_{V_r(\bb R^d)\times\bb R^r} \bb E[f(QU_t^{u,y},QY_t^{u,y})]\,\pi^*(\rrm du,\rrm dy)\,\rrm dQ\\
 & =\int_{O(r)}\int_{V_r(\bb R^d)\times\bb R^r}\bb E[f(U_t^{Qu,Qy},Y_t^{Qu,Qy})]\,\pi^*(\rrm du,\rrm dy)\,\rrm dQ\\
 & =\int_{O(r)}\int_{V_r(\bb R^d)\times\bb R^r} f(Qu,Qy)\,\pi^*(\rrm du,\rrm dy)\,\rrm dQ\\
 & =\int f\,d\widetilde{\pi}^*,
\end{aligned}
\]
where $(U_0^{u,y},Y_0^{u,y})=(u,y)$. 

Next, we verify that $\widetilde\pi^*$ is $O(r)$-invariant. For any $R\in O(r)$ and any $\varphi\in C_{\mathrm{b}}(V_r(\bb R^r)\times\bb R^r)$, it follows from the definition of $\widetilde\pi^*$ that
$$
\begin{aligned}
    \int_{V_r(\bb R^r)\times\bb R^r}\varphi(Ru,Ry)\,\widetilde \pi^*(\rrm du,\rrm dy)&=\int_{O(r)}\int_{V_r(\bb R^r)\times\bb R^r}\varphi(RQu,RQy)\, \pi^*(\rrm du,\rrm dy)\,\rrm dQ\\&=\int_{O(r)}\int_{V_r(\bb R^r)\times\bb R^r}\varphi(Qu,Qy)\, \pi^*(\rrm du,\rrm dy)\,\rrm dQ\\&= \int_{V_r(\bb R^r)\times\bb R^r}\varphi(u,y)\,\widetilde \pi^*(\rrm du,\rrm dy).
\end{aligned}
$$
Let $\widetilde{\mu}^*$ be the second marginal of $\widetilde{\pi}^*$. Then, for
every $\varphi\in C_{b}(\mathbb{R}^{r})$, 
\[
\int_{\bb R^r}\varphi(y)\,\tilde{\mu}^*(\rrm dy)=\int_{O(r)}\int_{\bb R^r}\varphi(Qy)\,\mu^*(\rrm dy)\,\rrm dQ.
\]
Moreover, 
\[
C_{\widetilde{\mu}^*}=\int_{O(r)}QC_{\mu^*}Q^{\top}\,\rrm dQ=\frac{\operatorname{Tr}(C_{\mu^*})}{r}I_{r}.
\]
Since $\det{C_{\mu^*}}\ne0$, we have $\frac{\operatorname{Tr}(C_{\mu^*})}{r}\ne0$,
and hence 
$
\det(C_{\widetilde{\mu}^*})\ne 0.$

\end{proof}
From the preceding discussion, it follows that both the first and second marginals of $\widetilde\pi^*$ are $O(r)$-invariant.

Next, we give the invariant probability of DLRA of system \eqref{eq.sec4}.
\begin{corollary}
    Under the same assumptions as in Theorem~\ref{inv}, the DLRA~\eqref{eqn:intro_xr} corresponding to Eqn.~\eqref{eq.sec4} admits an invariant probability measure.
\end{corollary}
\begin{proof}
{From Theorem~\ref{inv} the strong DO solution~\eqref{eq:DO} to Eqn.~\eqref{eq.sec4} admits an invariant probability measure $\pi^*$, and from Theorem~\ref{cor:inv} an invariant probability measure \(\widetilde\pi^*\), that is \(O(r)\)-invariant of 
the joint process \((U_t,Y_t)\), can be constructed from $\pi^*$.}
Let
\[
F:V_r(\mathbb R^d)\times\mathbb R^r\to\mathbb R^d,
\qquad
F(u,y):=u^\top y,
\]
and define the image measure
\(
\rho^*:=F_\#\widetilde\pi^*,
\)
where $$F_\#\widetilde\pi^*(A)=\widetilde\pi (\{(u,y)\in V_r(\mathbb R^d)\times\mathbb R^r: F(u,y)\in A\}),\quad A\in\scr B(\bb R^d).$$
Therefore, by the change-of-variables formula for image measures, for every bounded measurable \(\varphi:\mathbb R^d\to\mathbb R\),
\[
\int_{\mathbb R^d}\varphi(x)\,\rho^*(\rrm dx)
=
\int_{\mathbb R^d}\varphi(x)\,(F_\#\widetilde\pi^*)(\rrm dx)
=
\int_{V_r(\mathbb R^d)\times\mathbb R^r}
(\varphi\circ F)(u,y)\,\widetilde\pi^*(\rrm du,\rrm dy).
\]
Since $\varphi\circ F$ is bounded and measurable on $V_r(\mathbb R^d)\times\mathbb R^r$ and $\widetilde\pi$ is an invariant probability measure for the process $(U_t,Y_t)$, it yields that 
\[
\begin{aligned}
\mathbb E[\varphi(X_t^{(r)})]
=
\mathbb E[(\varphi\circ F)(U_t, Y_t)]
&=
\int_{V_r(\mathbb R^d)\times\mathbb R^r}
(\varphi\circ F)(u,y)\,\widetilde\pi^*(\rrm du,\rrm dy)\\
&=
\int_{\mathbb R^d}\varphi(x)\,\rho^*(\rrm dx).
\end{aligned}
\]

{Let 
\[
F:V_r(\mathbb R^d)\times \mathbb R^r\to \mathbb R^d,
\qquad
F(u,y):=u^\top y.
\]
Define
\[
\rho^*:=F_\#\widetilde\pi^*.
\]
For any \(Q\in O(r)\), define
\[
T_Q(u,y):=(Qu,Qy).
\]
Then
\[
(F\circ T_Q)(u,y)
=
F(Qu,Qy)
=
(Qu)^\top(Qy)
=
u^\top y
=
F(u,y).
\]
Hence
\[
F\circ T_Q=F.
\]
Therefore,
\[
F_\#\big((T_Q)_\#\widetilde\pi^*\big)
=
(F\circ T_Q)_\#\widetilde\pi^*
=
F_\#\widetilde\pi^*
=
\rho^*.
\]
Since \(\widetilde\pi^*\) is \(O(r)\)-invariant, \((T_Q)_\#\widetilde\pi^*=\widetilde\pi^*\). Thus \(\rho^*\) is compatible with the \(O(r)\)-symmetry of the DO factorisation.}

\end{proof}
This corollary provides a method for finding an invariant probability measure for the DLRA. In certain special cases, this method can be simplified. For example, when $(U_t,Y_t)$ possesses an invariant probability measure of the form $(\delta_{\phi^*},\mu^*),$ the invariant probability measure of $X_t^{(r)}$ is the image measure of $\mu^*$ under the linear map $y\mapsto \phi^{*\top }y$, i.e.,
$$
\nu := (\phi^{*\top})_\# \mu^*.
$$

Based on the distance between the SDE~\eqref{eq.sec4} and its DLRA~\eqref{eqn:intro_xr} established in Section~\ref{sec:uniform}, the next corollary estimates the distance between their invariant probability measures.
\begin{corollary}
    Under the same assumptions as in Theorem~\ref{maintheorem}. Assume moreover that the original SDE $(X_t)_{t\ge0}$ admits a unique invariant probability measure $\pi_1$, and that the reduced process $(X_t^{(r)})_{t\ge0}$ also possesses a unique invariant probability measure $\pi_2$. Then, we obtain
    $$\bb W_2\big(\pi_1,\pi_2\big)\le C(d-r).$$
\end{corollary}
\begin{proof}
   Set $\scr L_{X_0} = \pi_1$ and $\scr L_{X_0^{(r)}} = \pi_2$. By the definition of an invariant probability measure, we have $\scr L_{X_t} = \pi_1$ and $\scr L_{X_t^{(r)}} = \pi_2$ for all $t>0$. By Theorem~\ref{maintheorem}, there exists a constant $C>0$ such that 
   \begin{equation}
       \bb W_2\big(\pi_1,\pi_2\big)=\bb W_2^2(\scr L_{X_t},\scr L_{X_t^{(r)}})\le \rrm e^{\delta t}\bb W_2\big(\pi_1,\pi_2\big)+C(d-r),
   \end{equation}
   where $\delta=2(L_{b,3}+L_{\sigma,1})$. Then, letting $t\to+\infty$, we obtain the desired result.
\end{proof}

The following example illustrates that, even though 
$U$ evolves deterministically, the first marginal of the invariant measure need not be a Dirac measure. 
This depends on the structure of the deterministic flow.
		\begin{example}\label{exam:notdirac}
		Consider the following SDE in $\bb R^d$, 
        \begin{equation}
            \rrm dX_t=(A-I_d)X_t\, \rrm dt+\rrm d B_t,
        \end{equation}   where $A$ is an anti-symmetric orthogonal matrix on $\bb R^{d\times d}$, and $(B_t)_{t\ge0}$ is a $d$-dimensional Brownian motion. Then, this SDE admits a unique invariant probability measure given by $\al N(0,\frac{1}{2}I_d)$, i.e., the Gaussian distribution with mean $0$ and variance matrix is $\frac{1}{2}I_d$, by~\cite[Proposition 3.5]{pavliotis2014stochastic}. Moreover, the $1$-dimensional DO approximation of this SDE has an invariant probability $\pi=\nu\otimes\mu$, where $\nu$ is the uniform distribution on the unit sphere 
        \begin{equation}\label{eqn:S_1}
            \al S^{d-1}:=\{u_0\in\bb R^d:|u_0|=1\}, 
        \end{equation}and $\mu$ is the Gaussian distribution with mean $0$ and variance $1/2$. Furthermore, the $1$-dimensional DO approximation of this SDE also has an invariant probability $\pi^*$, which is $O(1)$-invariant, such that for every bounded measurable function $f:\bb R^d\times \bb R\to\bb R$, one has 
        \begin{equation}
            \int_{\bb R^d\times \bb R}f(u,y)\pi^*(\rrm du,\rrm dy)=\frac{1}{2}\Big(\int_{\bb R^d\times \bb R}f(u,y)\pi(\rrm du,\rrm dy)+\int_{\bb R^d\times \bb R}f(-u,-y)\pi(\rrm du,\rrm dy)\Big).
        \end{equation}
        \end{example} 
    
         Note that $UAU^\top=0$ and thus $UAU^\top U=0$ for any $U\in \mathbb{R}^{1\times d}$ because of $A$ being anti-symmetric. The pair $(U_t,Y_t)$ defined in Eqn.~\eqref{eq:DO}, where $(Y_t)_{t\ge0}$ is a $\bb R$-valued stochastic process, is therefore given by 
		\begin{equation}\label{exam.2}
			\begin{aligned}
				&\rrm dU_t=U_tA\rrm dt,\\
				&\rrm dY_t=-Y_t\rrm dt+U_t\rrm dB_t,
			\end{aligned}
		\end{equation}
		where $(B_t)_{t\ge0}$ is a $\bb R^d$-valued Brownian motion.
	From the explicit form of the dynamics, we obtain the solution
	\begin{equation}
		\begin{aligned}
			&U_t=U_0\rrm e^{At}=U_0\big(I_d\cos t+A\sin t\big),\\
			&Y_t=\rrm e^{-t}Y_0+\int _0^t\rrm e^{-(t-s)}U_s\rrm dB_s,
		\end{aligned}
	\end{equation}
	where the initial condition $(U_0,Y_0)$ satisfies the initial condition of strong DO solution introduced in Section~\ref{sec.intr}.  Note that $U_t=U_0\big(I_d\cos t+A\sin t\big)$, which implies that the process $(U_t)_{t\ge0}$ does not converge to a single point. Nevertheless, we can explicitly construct an invariant probability measure for $(U_t)_{t\ge0}$ as well as for the joint process $(U_t,Y_t)_{t\ge0}$, in accordance with Definition~\ref{def:inv}.
    
    Define 
	$$\gamma_t:=\int_0^tU_s \rrm dB_s.$$
	By L\'evy's characterization theorem~\cite[Chapter 3, Theorem 3.16]{karatzas2014brownian}, we know $(\gamma_t)_{t\ge0}$ is $\bb R$-valued standard Brownian motion. Then the process $Y$ can be rewritten as 
	\begin{equation}
		Y_t=\rrm e^{-t}Y_0+\int _0^t\rrm e^{-(t-s)}\rrm d\gamma_s.
	\end{equation}
	The process $(Y_t)_{t\ge0}$	is therefore an Ornstein-Uhlenbeck process with invariant probability measure $\mu$,  a Gaussian distribution with mean $0$ and variance $\frac{1}{2}$.	Consequently, for every bounded measurable function $g:\bb R\to\bb R$ and any initial data $Y_0=y\in\bb R$,
	\begin{align*}
		\int_{\bb R}\bb E[g(Y_t^{y})]\mu(\rrm dy)=\mu\Big( g\big(Y_t^{\cdot}\big)\Big)=\mu\big(g(\cdot)\big)=
		\int_{\bb R}g(y)\mu(\rrm dy).
	\end{align*}
	Define the product measure $\pi=\nu\otimes\mu$, where $\nu$ denotes the uniform distribution on the unit sphere $\al S^{d-1}$. Indeed,  by Theorem 2.2.1 in~\cite{chikuse2003statistics}, let $Z$ be an $\bb R^d$-valued Gaussian random vector with mean zero and identity covariance matrix $I_d$. Define $\xi$ by 
    \begin{equation}
        \xi=Z^\top(ZZ^\top)^{-\frac{1}{2}}.
    \end{equation} Then the distribution $\nu$ of $\xi$ is exactly the uniform distribution on the unit sphere $\al S^{d-1}$.
    
    For every bounded measurable function $f:\bb R^{d}\times\bb R\to\bb R$, we compute
	\begin{align*}	\pi\Big(f\big((U_t,Y_t)^\cdot\big)\Big)&=\int_{\al S^{d-1}}\int_{\bb R}\bb E[f(U_t^u,Y_t^{y})]\mu(\rrm dy)\nu(\rrm du)\\&=
		\int_{\al S^{d-1}}\int_{\bb R}\bb E[f(U_t^u,y)]\mu(\rrm dy)\nu(\rrm du)\\&=
		\int_{\bb R}\int_{\al S^{d-1}}f\big(u\rrm e^{At},y\big)\nu(\rrm du)\mu(\rrm dy).
	\end{align*}
	Since $\rrm e^{At}$ is an orthogonal matrix with determinant $1$, and the map $u\mapsto u\rrm e^{At} $ preserves uniform distribution and leaves the set $\al S^{d-1}$ invariant. Therefore, 
	$$\int_{\bb R}\int_{\al S^{d-1}}f\big(u\rrm e^{At},y\big)\nu(\rrm du)\mu(\rrm dy)=
	\int_{\bb R}\int_{\al S^{d-1}}f\big(u,y\big)\nu(\rrm du)\mu(\rrm dy)= \pi\big(f(\cdot)\big).$$
	
	Thus, the product measure $\pi=\nu\otimes\mu$ is the invariant probability measure for the system~\eqref{exam.2}. Next, we prove $\pi^*$ is also an invariant measure. Also using the same method of proving $\pi$ is an invariant probability measure, we have 
    \begin{align*}
    \pi^*\Big(f\big((U_t,Y_t)^\cdot\big)\Big)&=\frac{1}{2}\int_{\al S^{d-1}}\int_{\bb R}\bb E[f(U_t^u,Y_t^{y})]\mu(\rrm dy)\nu(\rrm du)+\frac{1}{2}\int_{\al S^{d-1}}\int_{\bb R}\bb E[f(U_t^{-u},Y_t^{-y})]\mu(\rrm dy)\nu(\rrm du)\\&=
		\frac{1}{2}\int_{\al S^{d-1}}\int_{\bb R}f(U_t^u,y)\mu(\rrm dy)\nu(\rrm du)+\frac{1}{2}\int_{\al S^{d-1}}\int_{\bb R}f(U_t^{-u},-y)\mu(\rrm dy)\nu(\rrm du)\\&=
		\frac{1}{2}\int_{\bb R}\int_{\al S^{d-1}}f\big(u\rrm e^{At},y\big)\nu(\rrm du)\mu(\rrm dy)+\frac{1}{2}\int_{\bb R}\int_{\al S^{d-1}}f\big(-u\rrm e^{At},-y\big)\nu(\rrm du)\mu(\rrm dy)\\&=\frac{1}{2}\int_{\bb R}\int_{\al S^{d-1}}f\big(u,y\big)\nu(\rrm du)\mu(\rrm dy)+\frac{1}{2}\int_{\bb R}\int_{\al S^{d-1}}f\big(-u,-y\big)\nu(\rrm du)\mu(\rrm dy)\\&=\pi^*(f).
	\end{align*}
    Example~\ref{exam:notdirac} shows that the first marginal of the invariant measure to Eqn.~\eqref{eq:DO} is not necessarily a Dirac measure. By contrast, the following example illustrates that such a marginal can indeed be a Dirac measure.
	\begin{example}
	    Consider the following SDE in $\bb R^d$,
        \begin{equation}
            \rrm d X_t=f(X_t)X_t\,\rrm dt+\sigma(X_t)\,\rrm d B_t.
        \end{equation}
        Here $f:\bb R^d\to\bb R$ is a negative bounded function, and $(B_t)_{t\ge0}$ denotes a $d$-dimensional Brownian motion. In addition, $\sigma$ satisfies Assumptions (H3) and (H5). If $\sup_{x\in\bb R^d} f(x)+L_{\sigma,1}<0$, where $L_{\sigma,1}$ is defined as in (H3), then the SDE admits a unique invariant probability measure~\cite[Theorem 3.7]{khasminskii2011stochastic}. Furthermore, the $r$-dimensional DO approximation of this SDE possesses an invariant probability measure of the form $\pi=\delta_u\otimes\mu_u$, where $\delta_u$ is a Dirac measure on $u\in V_r(\bb R^d)$defined by
\begin{equation}\label{eqn:dirac}
    \delta_{u}(A) := \left\lbrace\begin{aligned}
        & 0, && u \notin A;\\
        & 1, && u \in A.
    \end{aligned}\right.
\end{equation}
and $\mu_u$ is the invariant probability measure of $Y_t$, which depends on $u$.
	\end{example}
    Since $f$ is a bounded function and $U_t^\top(I_d-P_{U_t})=0$, we have
    $$\bb E\big[Y_t (f(U_t^\top Y_t)U_t^\top Y_t)^\top\big](I_d-P_{U_t})=\bb E\big[f(U_t^\top Y_t)Y_tY_t^\top \big]U_t^\top(I_d-P_{U_t})=0.$$ Consequently, the pair $(U_t,Y_t)$ defined in Eqn.~\eqref{eq:DO} satisfies
    \begin{equation}\label{eqn:exam3}
        \begin{aligned}
            C_{Y_t}\,\rrm dU_t&=0\,{\rrm dt};\\
            \rrm d Y_t&=f(U_t^\top Y_t) Y_t\,\rrm dt+U^\top_t\sigma(U_t^\top Y_t)\,\rrm d B_t.
        \end{aligned}
    \end{equation}
    Since $\sigma$ satisfies (H5), by the same arguments as in Theorem~\ref{theo:inverse}, it follows that $C_{Y_t} \neq 0$ for all $t \ge 0$. Consequently, $U_t$ remains constant in time. Hence, $Y_t$ is a Markov process parameterized by the initial value $U_0$, and it admits an invariant probability measure $\mu_{u}$. 

    It then follows that system \eqref{eqn:exam3} possesses an invariant probability measure of the form $\pi=\delta_u\otimes\mu_u$ for any $u\in V_r(\bb R^d)$.
        \\
    \\

	    To prove Theorem~\ref{inv}, we first decouple the distribution. We then transform Eqn.~\eqref{eq:DO} into a classical SDE. By the Krylov–Bogoliubov theorem (see Corollary 3.1.2 in~\cite{da1996ergodicity}), this classical SDE admits an invariant probability measure. Finally, using the Kakutani-Fan-Glicksberg fixed point theorem (see Theorem~\ref{fixedpoint}), we establish that the strong DO solution of Eqn.~\eqref{eq:DO} possesses an invariant probability measure. 
        
        For any $\kappa>0$ and any $p\ge 2$, fix $\mu\in \scr P_{2}^{\kappa}(\bb R^r)$. For any $\nu \in\scr P_{p}(\bb R^r)$ and $\eta\in\scr P(V_r(\bb R^d))$, consider the frozen DO system
         
         \begin{equation}\label{eq.FrozenDO}
             \begin{aligned}
                 C_{\mu}\rrm dU_t^{(\mu)}&=\int_{\bb R^r}x b(U_t^{(\mu)\top} x)^\top\mu(\rrm dx)(I_{d}-P_{U^{(\mu)}_t})\rrm dt,\\
				\rrm d Y_t^{(\mu)}&=U_t^{(\mu)}b(U_t^{(\mu)\top} Y^{(\mu)}_t)\rrm dt+U_t^{(\mu)}\sigma(U_t^{(\mu)\top} Y^{(\mu)}_t)\rrm dB_t,
             \end{aligned}
         \end{equation}
		with initial condition $\scr L_{(U_0^{(\mu)},Y^{(\mu)}_0)}=(\eta,\nu)$, where $P_{U^{(\mu)}_t}=U_t^{(\mu)\top}U_t^{(\mu)}.$
        Observe that Eqn.~\eqref{eq.FrozenDO} admits a unique solution $(U_t^{(\mu)}, Y_t^{(\mu)})_{t\ge0}$. Indeed, since the map $\psi:\bb R^{r\times d}\to \bb R^{r\times d}$ defined by  $$\psi(A) := \int_{\bb R^r} x \, b(A^{\top} x)^\top \, \mu(\rrm dx) \, (I_d - P_A),$$ is locally Lipschitz and has linear growth, and since $b,\sigma$ satisfy Assumptions~(H3), (H7) and (H8), by following the same arguments as in \cite[Theorems~3.4 and~3.6]{DLRASDE} we conclude that Eqn.~\eqref{eq.FrozenDO} possesses a unique solution. 
        
        We will employ the following lemma to establish that the collection $\big(\scr L_{(U_t^{(\mu)},Y_t^{(\mu)})}\big)_{t\ge0}$ of laws is tight.
        
		\begin{lemma}\label{bound}
			Suppose Assumptions (H3), (H7) and (H8) hold.
			Moreover, suppose  Assumption~\ref{AS:relation2} holds. Furthermore, suppose the initial condition $(\zeta,\xi)\in L^p(\Omega,V_r(\bb R^d)\times\bb R^r)$. Then, there exists a constant $C(b,\sigma,p)>0$, one has
             \begin{align}\label{eq.ubound}
                \bb E[\|U_t^{(\mu)}\|_{\rrm {HS}}^p]=r^\frac{p}{2},
            \end{align}
             and
			\begin{align}\label{eq.ybound}
				\bb E[|Y_t^{(\mu)}|^p]\le \rrm e^{-\lambda \frac{p}{2}t}\bb E[|\xi|^p]+\frac{2C(b,\sigma,p)}{\lambda p},
			\end{align}   
            where $\lambda=2L_{b,7}-2(p-1)L_{\sigma,1}-\frac{2}{p}$.
		\end{lemma}
		\begin{proof}
			Note that the function $U^{(\mu)}$ is absolutely continous on $[0,T]$, so it is differentiable almost everywhere. The derivative $\frac{\rrm dU^{(\mu)}_t}{\rrm dt}$ satisfies 
			\begin{equation}
				\frac{\rrm dU^{(\mu)}_t}{\rrm dt}U_t^{{(\mu)}\top}=C_{\mu}^{-1}\int_{\bb R^r}x b(U_t^{{(\mu)}\top} x)^\top\mu(\rrm dx)(I_{d}-P_{U^{(\mu)}_t})U^{{(\mu)}\top}_t=O_r,
			\end{equation}
			where $O_r$ is zero matrix. Thus $\frac{\rrm d}{\rrm dt}\big(U^{(\mu)}_tU_t^{{(\mu)}\top}\big)=O_r$. Therefore, by Eqn.~\eqref{eq:Vr}, it follows that
			\begin{equation}\label{eq.Unorm}
				\bb E[\|U_t^{(\mu)}\|^p_\rrm{HS}]=\bb E[\|U_0^{(\mu)}\|^p_\rrm{HS}]=\bb E[\|\zeta\|^p_\rrm{HS}]=r^\frac{p}{2}.
			\end{equation}
			 Similar to the proof of Theorem~\ref{bounded}, we apply It\^o's formula, (H3) and (H8)  to obtain 
             \begin{equation}\label{eqn:Ynorm}
                 \begin{aligned}
                  \rrm d\big(\rrm e ^{\lambda\frac{p}{2} t}|Y_t^{(\mu)}|^p\big)&\le \lambda\frac{p}{2}\rrm e ^{\lambda\frac{p}{2} t}|Y_t^{(\mu)}|^p\rrm dt+p\,\rrm e ^{\lambda\frac{p}{2} t}\big|Y_t^{(\mu)}\big|^{p-2}\big\langle Y_t^{(\mu)},U_t^{(\mu)}b(U_t^{(\mu)\top}Y^{(\mu)}_t)\big\rangle\rrm dt\\&\quad+\frac{(p-1)p}{2}\rrm e ^{\lambda\frac{p}{2} t}\big|Y_t^{(\mu)}\big|^{p-2}\big\|U_t^{(\mu)}\sigma(U_t^{(\mu)\top} Y^{(\mu)}_t)\big\|^2_\rrm{HS}\rrm dt+\rrm dM^{(\mu)}_t\\&
				\le \lambda\frac{p}{2}\rrm e ^{\lambda\frac{p}{2} t}|Y_t^{(\mu)}|^p\rrm dt-\frac{p}{2}\rrm e ^{\lambda\frac{p}{2} t}\big|Y_t^{(\mu)}\big|^{p-2}\Big(2L_{b,7}\big|Y_t^{(\mu)}\big|^2-2(p-1)L_{\sigma,1}\big|Y_t^{(\mu)}\big|^2-C(b,\sigma,p)\Big)\rrm dt+\rrm dM^{(\mu)}_t\\&\le
				-\rrm e ^{\lambda\frac{p}{2} t}\big|Y_t^{(\mu)}\big|^{p}\rrm dt+\frac{p}{2}\rrm e ^{\lambda\frac{p}{2} t}C(b,\sigma,p)\big|Y_t^{(\mu)}\big|^{p-2}\rrm dt+\rrm dM^{(\mu)}_t, 
                 \end{aligned}
             \end{equation}
             where $$M^{(\mu)}_t=p\int_0^t\rrm e ^{\lambda\frac{p}{2} s}\big|Y_s^{(\mu)}\big|^{p-2}\big\langle Y_s^{(\mu)}, U_s^{(\mu)}\sigma(U_s^{(\mu)\top} Y^{(\mu)}_s)\rrm dB_s\big\rangle.$$ Applying Young's inequality in case $p>2$, we obtain
             \begin{equation}
                 \frac{p}{2}C(b,\sigma,p)\big|Y_t^{(\mu)}\big|^{p-2}\le \big|Y_t^{(\mu)}\big|^{p}+\frac{2}{(p-2)}C(b,\sigma,p).
             \end{equation}
             When $p=2$, we get 
             $
                 \frac{p}{2}C(b,\sigma,p)\big|Y_t^{(\mu)}\big|^{p-2}=C(b,\sigma,2).
             $
            For any $N\in\bb N$, define $\tau_N=\inf\{t\ge0: |Y^{(\mu)}_t|>N\}$. Then
             we have 
             \begin{equation}
                 \rrm d\bb E[\rrm e ^{\lambda\frac{p}{2} (t\wedge\tau_N)}|Y_{t\wedge\tau_N}^{(\mu)}|^p]\le\bb E[\rrm e ^{\lambda\frac{p}{2}(t\wedge\tau_N)}]C(b,\sigma,p)\rrm dt\le \rrm e ^{\lambda\frac{p}{2}t}C(b,\sigma,p)\rrm dt.
             \end{equation}
            By Fatou's lemma, we have 
				\begin{align*}
					\rrm e ^{\lambda\frac{p}{2}t}\bb E[|Y_t^{(\mu)}|^p]\le \liminf_{N\to\infty}\bb E[\rrm e ^{\lambda\frac{p}{2}(t\wedge\tau_N)}]|Y_{t\wedge\tau_N}^{(\mu)}|^p]\le \bb E[|\xi|^p]+\frac{2C(b,\sigma,p)}{\lambda p}.
				\end{align*}
            Thus, we obtain the desired estimate
            \begin{equation}
                \bb E[|Y_t^{(\mu)}|^p]\le \rrm e^{-\lambda \frac{p}{2}t}\bb E[|\xi|^p]+\frac{2C(b,\sigma,p)}{\lambda p}.
            \end{equation}
		\end{proof}
        {
        \begin{remark}
        Assumption~\ref{AS:relation2} guarantees that \(\lambda>0\). Consequently, Eqn.~\eqref{eq.ybound} yields a uniform-in-time \(p\)-moment bound for \(Y_t\). Together with the identity
\[
\|U_t\|_{HS}^2=r,
\]
this implies
\[
\sup_{t\ge0}
\mathbb E\left[\|(U_t,Y_t)\|^p\right]
<\infty.
\]
Hence the family of laws
\[
\left\{\mathcal L(U_t,Y_t):t\ge0\right\}
\]
is tight. In particular, the associated occupation measures
\[
\pi_T
=
\frac1T\int_0^T \mathcal L(U_t,Y_t)\,dt
\]
are also tight. Combined with the Feller property established later in Lemma~\ref{lem,IMP}, this provides the hypotheses required for the Krylov--Bogoliubov theorem and ultimately yields the existence of an invariant probability measure.
        \end{remark}
        }

		\begin{lemma}\label{eq.continuous}
			Suppose Assumptions (H3), (H7) and (H8) hold. Moreover, suppose Assumption~\ref{AS:relation2} holds.
			Let $\nu_1,\, \nu_2\in\scr P_2^\kappa(\bb R^r)$, $(\xi,\zeta)\in V_r(\bb R^d)\times\bb R^r$ be arbitrary.
			Furthermore, let $(U^{(\nu_1)},Y^{(\nu_1)})$ and $(U^{(\nu_2)},Y^{(\nu_2)})$ be the solution of Eqn.~\eqref{eq.FrozenDO} having $(\xi,\zeta)$ as the initial condition, i.e., $(U_0^{(\nu_1)},Y_0^{(\nu_1)})=(U_0^{(\nu_2)},Y_0^{(\nu_2)})=(\xi,\zeta)$.
			Then, 	
			there exists a constant $C=C(\kappa,t,b,\sigma,\nu_1(|\cdot|^2),\nu_2(|\cdot|^2)>0$ such that
			\begin{equation}\label{eqn:nu1nu2}
				\bb E\big[\|(U_t^{(\nu_1)},Y_t^{(\nu_1)})-(U_t^{(\nu_2)},Y_t^{(\nu_2)})\|^2\big]\le C (1+|\zeta|^2)
				\bb W_2^2(\nu_1,\nu_2),
			\end{equation}
			holds for all $t>0$.
		\end{lemma}
		\begin{proof}
		In order to derive Eqn.~\eqref{eqn:nu1nu2}, we establish the following preliminary estimates.
        Let $J_1$, $J_2$, $J_3$, $I_1$ and $I_2$ be defined by 
        \begin{equation}\label{eqns:J}
            \begin{aligned}
                &J_1:=\|C^{-1}_{\nu_1}-C^{-1}_{\nu_2}\|^2_\rrm{HS}\int_0^t\Big|\int_{\bb R^r}x b(U_s^{(\nu_1)\top} x)^\top\nu_1(\rrm dx)(I_{d}-P_{U^{(\nu_1)}_s})\Big|^2\rrm ds;\\
                &J_2:=\|C^{-1}_{\nu_2}\|_\rrm{HS}^2\int_0^t\Big|\Big(\int_{\bb R^r}x b(U_s^{(\nu_1)\top} x)^\top\nu_1(\rrm dx)-\int_{\bb R^r}x b(U_s^{(\nu_2)\top} x)^\top\nu_2(\rrm dx)\Big)(I_{d}-P_{U^{(\nu_1)}_s})\Big|^2\rrm ds;\\&
                J_3:=|C^{-1}_{\nu_2}|^2\int_0^t\Big\|\int_{\bb R^r}x b(U_s^{(\nu_2)\top} x)^\top\nu_2(\rrm dx)(P_{U^{(\nu_2)}_s}-P_{U^{(\nu_1)}_s})\big)\Big\|_\rrm{HS}^2\rrm ds;\\&
                I_1:=\big|U_t^{(\nu_1)}b(U_t^{(\nu_1)\top} Y^{(\nu_1)}_t)-U_t^{(\nu_2)}b(U_t^{(\nu_2)\top} Y^{(\nu_2)}_t)\big|^2;\\&
                I_2:=\big\|U_t^{(\nu_1)}\sigma(U_t^{(\nu_1)\top} Y^{(\nu_1)}_t)-U_t^{(\nu_2)}\sigma(U_t^{(\nu_2)\top} Y^{(\nu_2)}_t)\big\|_\rrm{HS}^2.
            \end{aligned}
        \end{equation}
        Since $\nu_1,\nu_2\in\scr P^\kappa_2(\bb R^r)$, where $\scr P_2^\kappa(\bb R^r)$ is given in Eqn.~\eqref{eqn:P_2k}, it holds that $$\|C^{-1}_{\nu_1}\|_{\bb R^r\to\bb R^r}\le \frac{1}{\kappa},\, \|C^{-1}_{\nu_2}\|_{\bb R^r\to\bb R^r}\le \frac{1}{\kappa}.$$ Moreover, we have the estimate $$\|C^{-1}_{\nu_1}-C^{-1}_{\nu_2}\|^2_\rrm{HS}\le\|C^{-1}_{\nu_1}\|_{\bb R^r\to\bb R^r}^2\|C^{-1}_{\nu_2}\|_{\bb R^r\to\bb R^r}^2\|C_{\nu_1}-C_{\nu_2}\|_\rrm {HS}^2,$$ 
		and $$\|C_{\nu_1}-C_{\nu_2}\|_\rrm {HS}^2\le\|C^{1/2}_{\nu_1}+C^{1/2}_{\nu_2}\|_{\bb R^r\to\bb R^r}^2\|C^{1/2}_{\nu_1}-C^{1/2}_{\nu_2}\|_\rrm {HS}^2.$$
        Furthermore, the first factor on the right-hand side can be bounded by
        $$\|C^{1/2}_{\nu_1}+C^{1/2}_{\nu_2}\|_{\bb R^r\to\bb R^r}^2\le2(\nu_1(|\cdot|^2)+\nu_2(|\cdot|^2)), $$
        which yields
        $$\|C_{\nu_1}-C_{\nu_2}\|_\rrm {HS}^2\le 2(\nu_1(|\cdot|^2)+\nu_2(|\cdot|^2))\bb W_2^2(\nu_1,\nu_2).$$
        Then, applying H\"older's inequality, (H7) and the projection property of $P_{U^{(\nu_1)}_s}$, we obtain 
        \begin{equation}\label{eq:J1}
            J_1\le \frac{2tL^2_{b,6}(\nu_1(|\cdot|^2)+\nu_2(|\cdot|^2))(\nu_1(|\cdot|^2)+1)^2}{\kappa^4}\bb W_2^2(\nu_1,\nu_2).
        \end{equation}
        
		For $J_2$, we proceed by using the definition of coupling. Let $\pi\in\Pi(\nu_1,\nu_2)$ be an optimal coupling \cite[Theorem 6.2.4]{ambrosio2005gradient}. Then there exist random variables $Z_1$ and $Z_2$ such that $\scr L_{Z_1}=\nu_1, \scr L_{Z_2}=\nu_2$ and 
       $
            \bb W_2^2(\nu_1,\nu_2)=\bb E[|Z_1-Z_2|^2].
       $
        Therefore, (H7) and H\"older's inequality yield that
        \begin{equation}
            \begin{aligned}
                &\Big|\int_{\bb R^r}x b(U_s^{(\nu_1)\top} x)^\top\nu_1(\rrm dx)-\int_{\bb R^r}x b(U_s^{(\nu_2)\top} x)^\top\nu_2(\rrm dx)\Big|\\&=\big|\bb E\big[Z_1b(U_s^{(\nu_1)\top}Z_1)^\top-Z_2b(U_s^{(\nu_2)\top}Z_2)^\top\big]\big|\\&\le \bb E\big[|Z_1b(U_s^{(\nu_1)\top}Z_1)^\top-Z_2b(U_s^{(\nu_2)\top}Z_2)^\top|\big]\\&\le 
                \bb E\big[|Z_1|| b(U_s^{(\nu_1)\top}Z_1)-b(U_s^{(\nu_2)\top} Z_2)|\big]+\bb E\big[|Z_1-Z_2|| b(U_s^{(\nu_2)\top} Z_2)|\big]\\&\le 
                L_{b,6}\bb E\big[|Z_1|^2\big]|U_s^{(\nu_1)}-U_s^{(\nu_2)}|+ L_{b,6}\bb E\big[|Z_1||Z_1-Z_2|\big]+L_{b,6}\bb E\big[|Z_1-Z_2|(1+|Z_2|)\big]\\&\le L_{b,6}\nu_1(|\cdot|^2)\|U_s^{(\nu_1)}-U_s^{(\nu_2)}\|_\rrm{HS}+(L_{b,6}\nu_1(|\cdot|^2)^\frac{1}{2}+2L_{b,6}+2L_{b,6}\nu_2(|\cdot|^2)^\frac{1}{2})\bb W_2(\nu_1,\nu_2).
            \end{aligned}
        \end{equation}
        This, together with the second equation of Eqn.~\eqref{eqns:J}, implies
        \begin{equation}\label{eq:J2}
            \begin{aligned}
                J_2\le \frac{4}{\kappa^2r}\Big(L^2_{b,6}\nu_1(|\cdot|^2)^2\int_0^t\|U_s^{(\nu_1)}-U_s^{(\nu_2)}\|_\rrm{HS}^2\rrm ds+t(L^2_{b,6}\nu_1(|\cdot|^2)+4L^2_{b,6}+4L^2_{b,6}\nu_2(|\cdot|^2))\bb W^2_2(\nu_1,\nu_2)\Big).
            \end{aligned}
        \end{equation}
        Furthermore, from (H3) we obtain
        \begin{equation}\label{eq:J3}
            J_3\le \frac{L_{b,6}}{\kappa^2}(\nu_2(|\cdot|^2)+1)^2\int_0^t\|U_s^{(\nu_1)}-U_s^{(\nu_2)}\|_\rrm{HS}^2\rrm ds.
        \end{equation}
        Using assumptions (H3) and (H7), it follows that
        \begin{equation}\label{eq:lip}
            \begin{aligned}
                I_1=&\big|U_t^{(\nu_1)}b(U_t^{(\nu_1)\top} Y^{(\nu_1)}_t)-U_t^{(\nu_2)}b(U_t^{(\nu_2)\top} Y^{(\nu_2)}_t)\big|^2\\&\le 2\big|U_t^{(\nu_1)}-U_t^{(\nu_2)}\big|^2\big|b(U_t^{(\nu_1)\top} Y^{(\nu_1)}_t)\big|^2+2\big|U_t^{(\nu_2)}\big|^2\big|b(U_t^{(\nu_1)\top} Y^{(\nu_1)}_t-b(U_t^{(\nu_2)\top} Y^{(\nu_2)}_t)\big|^2\\&\le 4L^2_{b,6}\big(\big|Y^{(\nu_1)}_t\big|^2+1\big)\big|U_t^{(\nu_1)}-U_t^{(\nu_2)}\big|^2+2L^2_{b,6}\big|U_t^{(\nu_1)\top}Y^{(\nu_1)}_t-U_t^{(\nu_2)\top}Y^{(\nu_2)}_t\big|^2\\&\le 
                8L^2_{b,6}\big(\big|Y^{(\nu_1)}_t\big|^2+1\big)\big\|U_t^{(\nu_1)}-U_t^{(\nu_2)}\big\|^2_\rrm{HS}+4L^2_{b,6}\big|Y^{(\nu_1)}_t-Y^{(\nu_2)}_t\big|^2.
            \end{aligned}
        \end{equation}
		A completely analogous estimate holds for $I_2$, namely,  
        \begin{equation}\label{eqn:diff}
            \begin{aligned}
            I_2=&\big\|U_t^{(\nu_1)}\sigma(U_t^{(\nu_1)\top} Y^{(\nu_1)}_t)-U_t^{(\nu_2)}\sigma(U_t^{(\nu_2)\top} Y^{(\nu_2)}_t)\big\|_\rrm{HS}^2\\&\le 8L^2_{\sigma,1}\big(\big|Y^{(\nu_1)}_t\big|^2+1\big)\big\|U_t^{(\nu_1)}-U_t^{(\nu_2)}\big\|^2_\rrm{HS}+4L^2_{\sigma,1}\big|Y^{(\nu_1)}_t-Y^{(\nu_2)}_t\big|^2.
            \end{aligned}
            \end{equation}
       We now estimate Eqn.~\eqref{eqn:nu1nu2}. By adding and subtracting an appropriately chosen intermediate term, we decompose the original difference into three components. Combining Eqns.~\eqref{eq:J1}, \eqref{eq:J2} and \eqref{eq:J3}, we deduce that there exists a constant $C(\kappa,b,\nu_1(|\cdot|^2),\nu_2(|\cdot|^2))>0$ such that
        \begin{align*}
            &\|U^{(\nu_1)}_t-U_t^{(\nu_2)}\|_\rrm{HS}^2\\&=\Big\|\int_0^t\Big(C^{-1}_{\nu_1}\int_{\bb R^r}x b(U_s^{(\nu_1)\top} x)^\top\nu_1(\rrm dx)(I_{d}-P_{U^{(\nu_1)}_s})-C^{-1}_{\nu_2}\int_{\bb R^r}x b(U_s^{(\nu_2)\top} x)^\top\nu_2(\rrm dx)(I_{d}-P_{U^{(\nu_2)}_s})\Big)\rrm ds\Big\|^2_\rrm{HS}\\&\le 3tJ_1+3tJ_2+3tJ_3\\&\le t^2C(\kappa,b,\nu_1(|\cdot|^2),\nu_2(|\cdot|^2))\bb W_2^2(\nu_1,\nu_2)+tC(\kappa,b,\nu_1(|\cdot|^2),\nu_2(|\cdot|^2))\int_0^t\|U_s^{(\nu_1)}-U_s^{(\nu_2)}\|_\rrm{HS}^2\rrm ds.
        \end{align*}
        By Gr\"onwall's lemma, we obtain
        \begin{equation}\label{eq:Ubou}
           \|U^{(\nu_1)}_t-U_t^{(\nu_2)}\|_\rrm{HS}^2\le C(\kappa,t,b,\nu_1(|\cdot|^2),\nu_2(|\cdot|^2))\bb W_2^2(\nu_1,\nu_2).
        \end{equation}
        By H\"older's inequality, It\^o's isometry, together with Eqns.~\eqref{eq:lip}, \eqref{eqn:diff} as well as \eqref{eq:Ubou}, there exists a constant $C(\kappa,t,b,\sigma,\nu_1(|\cdot|^2),\nu_2(|\cdot|^2))$ such that 
			\begin{equation*}
				\begin{aligned}
					\bb E[|Y_t^{(\nu_1)}-Y_t^{(\nu_2)}|^2]
                    &\le 2t\int_0^t\bb E\big[\big|U_s^{(\nu_1)}b(U_s^{(\nu_1)\top} Y^{(\nu_1)}_s)-U_s^{(\nu_2)}b(U_s^{(\nu_2)\top} Y^{(\nu_2)}_s)\big|^2\big]\rrm ds\\&\quad+2\int_0^t\bb E\big[\big\|U_s^{(\nu_1)}\sigma(U_s^{(\nu_1)\top} Y^{(\nu_1)}_s)-U_s^{(\nu_2)}\sigma(U_s^{(\nu_2)\top} Y^{(\nu_2)}_s)\big\|_\rrm{HS}^2\big]\rrm ds\\&\le C(\kappa,t,b,\sigma,\nu_1(|\cdot|^2),\nu_2(|\cdot|^2))\bb W_2^2(\nu_1,\nu_2)\Big(\int_0^t\bb E[\big|Y^{(\nu_1)}_s\big|^2]\rrm ds+1\Big)\\&\quad+4(L^2_{b,6}+L^2_{\sigma,1})\int_0^t\bb E\big[\big|Y^{(\nu_1)}_s-Y^{(\nu_2)}_s\big|^2\big]\rrm ds.
				\end{aligned}
			\end{equation*}
            From Lemma~\ref{bound}, Gr\"onwall's inequality implies that  there exists a constant $C(\kappa,t,b,\sigma,\nu_1(|\cdot|^2),\nu_2(|\cdot|^2))>0$ such that
            \begin{equation}
                \bb E[|Y_t^{(\nu_1)}-Y_t^{(\nu_2)}|^2]\le C(\kappa,t,b,\sigma,\nu_1(|\cdot|^2),\nu_2(|\cdot|^2))(1+|\zeta|^2)\bb W_2^2(\nu_1,\nu_2).
            \end{equation}
			
			Consequently, we arrive at the desired estimate 
			\begin{equation*}
				\bb E\big[\|(U_t^{(\nu_1)},Y_t^{(\nu_1)})-(U_t^{(\nu_2)},Y_t^{(\nu_2)})\|^2\big]\le C(\kappa,t,b,\sigma,\nu_1(|\cdot|^2),\nu_2(|\cdot|^2))(1+|\zeta|^2)\bb W_2^2(\nu_1,\nu_2),
			\end{equation*}
			which completes the proof.
		\end{proof}
		\begin{lemma}\label{le.feller}
			Assume Assumptions (H3), (H7) and  (H8) hold. Moreover, suppose Assumption~\ref{AS:relation2} holds. Let $\nu\in\scr P_2^\kappa(\bb R^r)$, $(\xi,\zeta),\,(\tilde \xi,\tilde \zeta)\in V_r(\bb R^d)\times\bb R^r$ be arbitrary. Furthermore,  let $(U^{{(\nu)}},Y^{(\nu)})\, (resp.\, (\widetilde U^{(\nu)},\widetilde Y^{(\nu)})\,)$ be the solutions to Eqn.~\eqref{eq.FrozenDO}	with initial condition $(\xi,\zeta)$ (resp. $(\tilde \xi,\tilde \zeta)$). Then, for every $t\ge0$, there exists a constant $C=C(\kappa,t,b,\sigma,\nu(|\cdot|^2))>0$ such that    
			$$\bb E\big[\|(U_{t}^{(\nu)},Y_t^{(\nu)})-(\widetilde U_t^{(\nu)},\widetilde Y_t^{(\nu)})\|^2\big]\le C(1+|\zeta|^2)\big(\|\xi-\tilde\xi\|_\rrm{HS}^2+|\zeta-\tilde\zeta|^2\big).$$ In particular, $(U_t^{(\nu)},Y_t^{(\nu)})_{t\ge0}$ is a Feller process on $V_r(\bb R^d)\times \bb R^r$ as defined in \cite[Page 20]{da1996ergodicity}. 
		\end{lemma}
		\begin{proof}
			Arguing as in the proof of Lemma~\ref{eq.continuous}, we obtain for any fixed $t\ge0$, there exists a constant $C=C(t,\kappa,b,\sigma,\nu(|\cdot|^2))>0$ such that
            \begin{equation}\label{eqn:Unu}
				\begin{aligned}
					\| U_{t}^{(\nu)}-\widetilde U_t^{(\nu)}\|_\rrm{HS}^2]\le C\|\xi-\tilde\xi\|_\rrm{HS}^2,
				\end{aligned}
			\end{equation}
            and 
			\begin{equation}\label{eqn:Ynu}
				\begin{aligned}
					\bb E[|Y_{t}^{(\nu)}-\widetilde Y_t^{(\nu)}|^2]\le C|\zeta-\tilde\zeta|^2+C\|\xi-\tilde\xi\|_\rrm{HS}^2\int_{0}^{t}(1+\bb E[|Y_{s}^{(\nu)}|^2])\rrm ds.
				\end{aligned}
			\end{equation}	
			Combining Lemma~\ref{bound} with Eqns.~\eqref{eqn:Unu} and \eqref{eqn:Ynu}, we obtain
			\begin{equation}
					\bb E\big[\|(U_{t}^{(\nu)},Y_t^{(\nu)})-(\widetilde U_t^{(\nu)},\widetilde Y_t^{(\nu)})\|^2\big]\le C(1+|\zeta|^2)\big(\|\xi-\tilde\xi\|_\rrm{HS}^2+|\zeta-\tilde\zeta|^2\big).
			\end{equation}
			  We conclude the desired result. 

            Besides, since $(U^{(\nu)}_t,Y^{(\nu)}_t)_{t\ge0}$ is a continuous Markov process, for any bounded continuous function $f:V_r(\bb R^d)\times \bb R^r\to\bb R$, let $(v_n,y_n)$ be a sequence in $V_r(\bb R^d)\times\bb R^r$ such that $(v_n,y_n)\to (v,y)$ as $n\to\infty$. Furthermore, let $(U^{{(\nu,n)}},Y^{(\nu,n)})$ (resp.$(U^{(\nu)},Y^{(\nu)})$) be the solutions to Eqn.~\eqref{eq.FrozenDO}	with initial data $(v_n,y_n)$ (resp. $(v,y)$). Then
            \begin{equation}
                \lim_{n\to\infty}\bb E\big[\|(U_{t}^{(\nu,n)},Y_{t}^{(\nu,n)})-(U_{t}^{(\nu)},Y_{t}^{(\nu)})\|^2\big]=0.
            \end{equation}
            It follows that $Y_{t}^{(\nu,n)}\to Y_t^{(\nu)}$ in distribution and $U_{t}^{(\nu,n)}\to U_t^{(\nu)}$. Since $f$ is bounded continuous on $V_r(\bb R^d)\times \bb R^r$, we have  
            \begin{align*}
                \lim_{n\to\infty}|\al P^{(\nu)}_tf(v_n,y_n)-\al P^{(\nu)}_tf(v,y)|=\lim_{n\to\infty}\bb E\big|f(U_t^{(\nu,n)},Y_t^{(\nu,n)})-f(U_t^{(\nu)},Y_t^{(\nu)})\big|=0.
            \end{align*}
            Therefore, $(U_t^{(\nu)},Y_t^{(\nu)})_{t\ge0}$ is a Feller process on $V_r(\bb R^d)\times \bb R^r$. 
		\end{proof}
        By Lemma~\ref{le.feller}, the process $(U^{(\mu)}_t, Y^{(\mu)}_t)_{t\ge0}$ is Feller on the space $V_r(\bb R^d)\times \bb R^r$ equipped with the norm~\eqref{eqn:norm}. We now set
        \begin{equation}\label{eqn:M}
            M:=\frac{2C(b,\sigma,p)}{\lambda p}+r^\frac{p}{2},
        \end{equation}
		where $C(b,\sigma,p)$ and $\lambda$ are constants appearing in Lemma~\ref{bound}. Recall that the space $\scr P^\kappa_2(\bb R^{r})$ is defined by 
        $$\scr P_{2}^{\kappa}(\bb R^r)=\Big\{\mu\in \scr P_2(\bb R^r): C_\mu\succeq \kappa I_{r\times r}\Big\},$$
        where 
        \begin{equation}\label{eqn:kappa}
            \kappa=\frac{L^2_{\sigma,3}}{4L_{b,6}(1+M^{2/p})^2}\wedge\frac{M^{2/p}}{2}>0.
        \end{equation}
        Define a subspace $\al K_M^\kappa$ of $\scr P_2^\kappa(\bb R^r)$ by 
        \begin{equation}\label{eqn:K}
            \al K^\kappa_M=\{\mu\in\scr P^\kappa_2(\bb R^r):\mu(|\cdot|^p)\le M\}.
        \end{equation}
        Note that $\al K^\kappa_M$ is not a vector space (i.e., it is not closed under addition or scalar multiplication). We have to suitably enlarge  the space $ \al K^\kappa_M$ such that the resulting space is a Banach space, and the restriction of the fixed point map coincides with a subset of $ \al K^\kappa_M$. Define 
		\begin{align*}
			\scr P_2:=\{\mu: \mu \text{ is a finite signed measure on  $\bb R^r$  satisfying }|\mu|(|\cdot|^2)<\infty \},
		\end{align*}
		where $|\mu|$ is the variation of the signed measure $\mu$. The space $\scr P_2$ is a Banach space under the Kantorovich-Rubinstein metric~\cite{bao2022existence}:
		\begin{align*}
			\al W_2\big(\mu_1,\mu_2\big)&:=\Big(\sup_{(f,g)\in \scr F_{\rrm{Lip}}(\bb R^r)}\Big|\int_{\bb R^r}f(x)\mu_1(\rrm dx)-\int_{\bb R^r}g(x)\mu_2(\rrm dx)\Big|\Big)^\frac{1}{2}, \text{ for any }\mu_1,\,\mu_2\in\scr P_2
		\end{align*}	
		where 
		\begin{align*}
			\scr F_{\mathrm{Lip}}
            (\bb R^r):=\{(f,g): f,g \text{ are Lipschitz functions on }\bb R^r \text{ such that }f(x)\le g(y)+ |x-y|^2, \text{ for all }x,y\in\bb R^r\}.
		\end{align*}
       When $(\scr P_2,\al W_2)$ is restricted to $\scr P_2(\bb R^r)$, the metrics $\al W_2$ and $\bb W_2$ coincide; that is,
        \begin{equation}
            \al W_2(\mu_1,\mu_2) = \bb W_2(\mu_1,\mu_2)
            \qquad \text{for any } \mu_1,\mu_2 \in \scr P_2(\bb R^r).
        \end{equation}
		  We now show that the frozen DO system~\eqref{eq.FrozenDO} admits an invariant probability measure.
		\begin{lemma}\label{lem,IMP}
			Suppose Assumptions (H3), (H7) and (H8) hold. Moreover, suppose Assumption~\ref{AS:relation2} holds. Then the process $(U^{(\mu)}_t, Y^{(\mu)}_t)_{t\ge0}$, governed by the frozen DO system ~\eqref{eq.FrozenDO}, admits an invariant probability measure.
		\end{lemma}
		\begin{proof}
            We now show that the family $\{\scr L_{(U_t^{(\mu)},Y^{(\mu)}_t)} : t \ge 0\}$ is tight. By Lemma~\ref{bound}, we have the uniform moment bound
$$
\sup_{t\ge0}\mathbb E\big[\|(U^{(\mu)}_t,Y^{(\mu)}_t)\|^p\big]\le M .
$$
Applying Chebyshev's inequality, it follows that for any $\varepsilon > 0$, one may choose $R > 0$ sufficiently large such that
			\begin{equation}
            \begin{aligned}
                \scr L_{(U^{(\mu)}_t, Y^{(\mu)}_t)}(\|(\cdot,\cdot)\|> R)&=\scr L_{(U^{(\mu)}_t, Y^{(\mu)}_t)}(\{(u,y)\in V_r(\bb R^d)\times \bb R^r:\|(u,y)\|>R\})\\&=\bb P(\|(U^{(\mu)}_t, Y^{(\mu)}_t)\|>R)\\&\le \frac{1}{R^p}\bb E\big[\|(U^{(\mu)}_t,Y^{(\mu)}_t)\|^p\big]\le \frac{M}{R^p}<\varepsilon,
            \end{aligned}
			\end{equation}
which holds uniformly in $t \ge 0$. Consequently, we obtain
\begin{equation}\label{eq:tight}
    \sup_{t\ge0}
\scr L_{(U^{(\mu)}_t,Y^{(\mu)}_t)}\big(\|(u,y)\|>R\big)<\varepsilon.
\end{equation}

This establishes the tightness of the family $\{\scr L_{(U^{(\mu)}_t, Y^{(\mu)}_t)};t\ge0\}$.
            
Let $(\delta_{u_0},\nu) \in \scr P(V_r(\bb R^d)) \times \mathcal K^\kappa_M$ be the initial distribution of $(U^{(\mu)}_0,Y^{(\mu)}_0)$, where $\delta_{u_0}$ denotes the Dirac measure at $u_0$ as defined in Eqn.~\eqref{eqn:dirac}, and let $P_t^{(\mu)}$ be the Markov semigroup generated by the process $(U^{(\mu)}_t,Y^{(\mu)}_t)$.
Define the family of probability measures 
$$
\pi^{(\mu)}_t:=\frac{1}{t}\int_0^t \al P_s^{*,(\mu)}(\delta_{u_0},\nu)\,\rrm ds,
$$
where $P_t^{*,(\mu)}$ denotes the adjoint of $P_t^{(\mu)}$. 
By Eqn.~\eqref{eq:tight} we have 
\begin{align*}
    \pi^{(\mu)}_t(\|(u,y)\|>R)=\frac{1}{t}\int_0^t\scr L_{(U^{(\mu)}_s,Y_s^{(\mu)})}(\|(u,y)\|>R)\, \rrm ds
    \le  \sup_{t\ge0}
\scr L_{(U^{(\mu)}_t,Y^{(\mu)}_t)}\big(\|(u,y)\|>R\big)<\varepsilon.
\end{align*}
Therefore, the family $\{\pi^{(\mu)}_t:t>0\}$ is also tight.
            By the Krylov--Bogoliubov theorem~\cite[Corollary 3.1.2]{da1996ergodicity}, since $(U^{(\mu)}_t, Y^{(\mu)}_t)$ generates a Markov semigroup $\al P_t^{(\mu)}$, there exists a sequence $t_n\uparrow+\infty$ such that the weak convergence
			\begin{equation}
				\pi^{(\mu)}_{t_n}\Rightarrow \pi^{(\mu)},\quad\text{ as } n\to \infty.
			\end{equation}
			Consequently, $\pi^{(\mu)}$ is an invariant probability measure for the semigroup
$P_t^{(\mu)}$.

		\end{proof}
For any $\mu\in\scr P_2^\kappa(\bb R^r)$, define the 
collection of invariant probability measures of the process $(U_t^{(\mu)}, Y_t^{(\mu)})_{t \ge 0}$, which solves Eqn.~\eqref{eq.FrozenDO}, by 
\begin{equation}\label{eqn:inv}
    \Lambda^\kappa_{\mu}=\Big\{\pi^{(\mu)}\in\scr P(V_r(\bb R^d)\times\bb R^r):\scr L_{(U_t^{(\mu)},Y_t^{(\mu)})}=\pi^{(\mu)},\, t\ge0\Big\},
\end{equation}

Next, define the second marginal (projection) of $\Lambda_\mu^\kappa$ by    
    \begin{equation}\label{eqn:Xi}
        \Xi^\kappa_\mu=\Big\{\nu\in \scr P^\kappa_2(\bb R^r): \text{there exists a }\pi\in\Lambda^\kappa_{\mu},\text{  s.t. }\nu(A)=\int_{V_r(\bb R^d)}\pi(\rrm dx,A),\text{ for any }A \in\scr B(\bb R^r)\Big\}.
    \end{equation}
By Lemma~\ref{lem,IMP}, we have $\Lambda^\kappa_{\mu} \neq \emptyset$, and consequently $\Xi^\kappa_\mu\neq\emptyset$. In the next lemma, we show that $\Xi_\mu^\kappa\subset\al K^\kappa_M$.
 
        \begin{lemma}\label{lem,bound}
			Suppose Assumptions (H3), (H5), (H7) and (H8) hold. Moreover, suppose Assumption~\ref{AS:relation2} holds. Then $\Xi_\mu^\kappa\subset\al K^\kappa_M$ for any $\mu\in\al K^\kappa_M$,
            where $\al K^\kappa_M$ is given in Eqn.~\eqref{eqn:K}.
		\end{lemma}
        \begin{proof}
		Fix $\nu\in \Xi^\kappa_\mu$ arbitrarily. Then, by the definition of $\Xi_\mu^\kappa$, there exists an invariant probability measure $\pi \in \Lambda^\kappa_{\mu}$ for the process $(U_t^{(\mu)},Y_t^{(\mu)})_{t\ge 0}$ whose second marginal coincides with $\nu$. 
        Since $\pi$ is invariant for the process $(U_t^{(\mu)},Y_t^{(\mu)})_{t\ge 0}$, 
for every bounded measurable function 
$f: V_r(\mathbb R^d) \times \mathbb R^r \to \mathbb R$ and every $t \ge 0$, we have
        \begin{equation}\label{eqn:finv}
            \int f(u,y)\,\pi(\rrm du,\rrm dy)
        =
        \int \mathbb E\big[f(U_t^{(\mu),u,y},Y_t^{(\mu),u,y})\big]\,
       \pi(\rrm du,\rrm dy),
        \end{equation}
        where $(U_0^{(\mu),u,y},Y_0^{(\mu),u,y})=(u,y)$.

            For any $N>0$ and $p\ge 2$, choose $$f(u,y)=|y|^p\wedge N,$$
			so that $f$ is bounded. Since the function $z \mapsto z \wedge N$ is concave on $[0,\infty)$, Jensen’s inequality yields 
            \begin{equation}\label{eq:piy}
                \begin{aligned}
                    \int_{\bb R^r}\big(|y|^p)\wedge N\big)\nu(\rrm dy)&=\int_{V_r(\bb R^d)\times \bb R^r}\big(|y|^p)\wedge N\big)\pi(\rrm du,\rrm dy)\\&=\int_{V_r(\bb R^d)\times \bb R^r}\bb E\big[(|Y_t^{(\mu),u,y}|^p)\wedge N\big]\pi(\rrm du,\rrm dy)\\&\le \int_{V_r(\bb R^d)\times \bb R^r}\big(\bb E[|Y_t^{(\mu),u,y}|^p]\wedge N\big)\pi(\rrm du,\rrm dy).
                \end{aligned}
            \end{equation}
            By Lemma~\ref{bound}, it follows that
            \begin{equation}\label{eq:uybo}
                \begin{aligned}
                    \bb E[|Y_t^{(\mu),u,y}|^p]\le \rrm e^{-\lambda t}|y|^p+\frac{2C(b,\sigma,p)}{\lambda p}.
                \end{aligned}
            \end{equation}
            
            Letting $t \to +\infty$ and applying the Lebesgue dominated convergence theorem, together with Eqns.~\eqref{eq:piy} and \eqref{eq:uybo}, we obtain
            
            $$\int_{\bb R^r}\big((|y|^p)\wedge N\big)\nu(\rrm dy)\le \frac{2C(b,\sigma,p)}{\lambda p}.$$
            Passing to the limit as $N \to \infty$ and invoking the monotone convergence theorem, it follows that
            \begin{equation}\label{eqn:Ybound}
                \int_{\bb R^r}|y|^p\nu(\rrm dy)\le \frac{2C(b,\sigma,p)}{\lambda p}\le M.
            \end{equation}
           
    Next, we will show  $ C_\nu\succeq\kappa I_r$. Consider the smallest eigenvalue of $C_{\nu}$. For any unit vector $e\in\bb R^r$ and $N>0$, consider a bounded function
    $$f(x,y)=(e^\top yy^\top e)\wedge N.$$
    By Eqn.~\eqref{eqn:finv}, we have
            \begin{align*}
                \int_{\bb R^r} \big((e^\top yy^\top e)\wedge N\big)\nu(\rrm dy)&=\int_{V_r(\bb R^d)\times\bb R^r}\big((e^\top yy^\top e)\wedge N\big)\pi(\rrm du, \rrm dy)\\&
                =\int_{V_r(\bb R^d)\times\bb R^r}\big(\bb E\big[ (e^\top Y^{(\mu)}_tY^{(\mu)\top}_te)\wedge N\big]\big)\pi(\rrm du, \rrm dy).
                \end{align*}
    Taking $N \to\infty$, and applying the monotone convergence theorem, one has 
    \begin{equation*}
        \int_{\bb R^r} (e^\top yy^\top e)\nu(\rrm dy)
                =\int_{V_r(\bb R^d)\times\bb R^r}\big(\bb E[ e^\top Y^{(\mu)}_tY^{(\mu)\top}_te]\big)\pi(\rrm du, \rrm dy).
    \end{equation*}
    Then, Theorem~\ref{theo:inverse} yields 
    \begin{align*}
                &\int_{\bb R^r} (e^\top yy^\top e)\nu(\rrm dy)\\&
                =\int_{V_r(\bb R^d)\times\bb R^r}\big(\bb E[ e^\top Y^{(\mu)}_tY^{(\mu)\top}_te]\big)\pi(\rrm du, \rrm dy)\\&\ge\int_{V_r(\bb R^d)\times\bb R^r} \Big(\rrm e^{-\frac{t}{\varepsilon}}e^\top\bb E[Y^{(\mu)}_0Y^{(\mu)\top}_0]e+\varepsilon\frac{L_{\sigma,3}}{2}\Big(1-\frac{\big(\bb E[|Y_0^{(\mu)}|^p\big)^{2/p}}{1+M^{2/p}}\Big)(1-\rrm e^{-\frac{t}{\varepsilon}})\Big)\pi(\rrm du, \rrm dy)\\&=\int_{V_r(\bb R^d)\times\bb R^r} \Big(\rrm e^{-\frac{t}{\varepsilon}}e^\top yy^\top e+\varepsilon\frac{L_{\sigma,3}}{2}\Big(1-\frac{|y|^2}{1+M^{2/p}}\Big)(1-\rrm e^{-\frac{t}{\varepsilon}})\Big)\pi(\rrm du, \rrm dy),
                \end{align*}
            where $\varepsilon$ is given in Theorem~\ref{theo:inverse}.
                
    Then, taking $t\to+\infty$ and using
    the Lebesgue dominated convergence theorem, one has
    \begin{align}\label{eqn:posi}
        \int_{\bb R^r} (e^\top yy^\top e)\nu(\rrm dy)\ge \int_{V_r(\bb R^d)\times\bb R^r} \varepsilon\frac{L_{\sigma,3}}{2}\Big(1-\frac{|y|^2}{1+M^{2/p}}\Big)\pi(\rrm du, \rrm dy).
    \end{align}
    From Eqn.~\eqref{eqn:Ybound}, we obtain
    \begin{equation*}
        -\int_{V_r(\bb R^d)\times\bb R^r}|y|^2\pi(\rrm du, \rrm dy)=-\int_{\bb R^r}|y|^2\nu(\rrm dy)\ge -M^{2/p},
    \end{equation*}
    which, combined with Eqn.~\eqref{eqn:posi}, yields
    \begin{equation}\label{eqn:semidef}
                e^\top \int_{\bb R^r} yy^\top \nu(\rrm dy)e\ge\frac{L^2_{\sigma,3}}{4L_{b,6}(1+M^{2/p})^2}\ge\kappa.
    \end{equation}
           From Eqns.~\eqref{eqn:Ybound} and \eqref{eqn:semidef} we conclude that
            $$\Xi_\mu^\kappa\subset \al K_M^\kappa.$$
        \end{proof}
We will employ the Kakutani-Fan-Glicksberg fixed point theorem to prove Theorem \ref{inv}.
    \begin{theorem}[Kakutani-Fan-Glicksberg]\cite[Theorem 1]{fan1952fixed}\label{fixedpoint}
    Let $\scr P_2$ be a locally convex 
    Hausdorff topological vector space, and let $\al K_M^\kappa\subset \scr P_2$  be a  nonempty, compact and convex set. Suppose $f:\al K_M^\kappa\to 2^{\al K_M^\kappa}$  be a set-valued function on $\al K_M^\kappa$ (where $2^{\al K_M^\kappa}$ is the set of subsets of $\al K_M^\kappa$) satisfying the following conditions:
    \begin{itemize}
        \item $f(x)$ is nonempty and convex for every $x\in \al K_M^\kappa$.
        \item The graph $Gr(f) := \{(x, y) \in \al K_M^\kappa \times \al K_M^\kappa: y \in f(x)\}$ is a closed subset of $\al K_M^\kappa \times \al K_M^\kappa$.
    \end{itemize}
    Then, $f$ admits a fixed point $x^*\in \al K_M^\kappa$,i.e., $x^*\in f(x^*)$.
    \end{theorem}
        \begin{proof}[The proof of theorem~\ref{inv}]
            Recall that $\Xi_\mu^\kappa$ is defined in Eqn.~\eqref{eqn:Xi}.       
            We now define the set-valued map $\gamma: \al K_M^\kappa \to 2^{\al K_M^\kappa}$ by 
			\begin{equation}\label{eqn:gamma}
				\gamma(\mu)=\Xi^\kappa_{\mu},\, \mu\in\al K^\kappa_M.
			\end{equation}
			Lemma~\ref{lem,bound} shows that $\Xi_\mu^\kappa\subset \al K_M^\kappa$, and hence $\gamma(\mu)\in  2^{K_M^\kappa}$.
			
			To complete the proof, it suffices to show that the set-valued map $\gamma$ admits a fixed point. 
Indeed, if $\gamma$ possesses a fixed point---that is, if there exists a probability measure $\mu \in \al K_M^\kappa$ such that $\mu \in \gamma(\mu) = \Xi_{\mu}^\kappa$---then, by the definition of $\Xi_{\mu}^\kappa$, there exists a probability measure $\pi \in \scr P(V_r(\bb R^d) \times \bb R^r)$ satisfying
$$
\pi = \scr L_{(U_t^{(\mu)}, Y^{(\mu)}_t)} \quad \text{and} \quad \scr L_{Y^{(\mu)}_t} = \mu .
$$ 
            Consequently, the strong solution of the DO system~\eqref{eq:DO} coincides with the frozen DO solution given by Eqn.~\eqref{eq.FrozenDO}. In particular, the strong DO dynamics~\eqref{eq:DO} admit an invariant probability measure $\pi$.
            
            To show that the set-valued map $\gamma$ admits a fixed point, we verify the hypotheses of the Kakutani--Fan--Glicksberg fixed point theorem (Theorem~\ref{fixedpoint}). Specifically, we must check the following three conditions 
			\begin{itemize}
                \item[(1)] The set $\al K_M^\kappa$ is compact and convex in the normed space $(\scr P_2, \al W_2)$, which is a  locally convex Hausdorff topological vector space.
				\item[(2)] $\Xi^\kappa_{\mu}$ is nonempty and convex for every $\mu\in\al K^\kappa_M$.
				\item[(3)] The graph of $\gamma$ defined in Eqn.~\eqref{eqn:gamma} is a closed subset of $\al K_M^\kappa\times\al K_M^\kappa$. 
			\end{itemize}
            
            We first verify the first condition of Theorem~\ref{fixedpoint}, namely that $\mathcal K_M^\kappa$ is compact and convex. Since $\al K_M^\kappa\subset \scr P_2(\bb R^r)$, we just prove The set $\al K_M^\kappa$ is compact and convex in the normed space $(\scr P_2(\bb R^r), \bb W_2)$.
            To this end, we apply~\cite[Theorem 5.5]{MR2091955}, which requires verifying the following conditions
			\begin{itemize}
				\item $\al K_M^\kappa$ is compact in the  topology generated by the open sets $$\{\nu\in \scr P(\bb R^r):|\nu(f)-\rho(f)|<\varepsilon\}, \text{ for any }\varepsilon>0,\, \rho\in\scr P(\bb R^r),\, f\in C_{\mathrm{b}}(\bb R^r).$$
				\item $\lim_{n\to \infty}\sup_{\mu\in\al K_M^\kappa}\int_{\{x\in\bb R^r:|x|\ge n\}}|x|^2\mu(\rrm dx)=0.$
			\end{itemize}
            Let $(\mu_n)_{n\ge0}$ be a sequence in $\al K_M^\kappa$, i.e.,
            $$\mu_n(|\cdot|^p)\le M,\text{ for all } n\ge 0.$$
			By Chebyshev's inequality, for any $\varepsilon>0$,  we choose $N$ sufficiently large such that 
			\begin{equation}
				\mu_n(|\cdot|> N)\le\frac{1}{N^p}\mu_n(|\cdot|^p)\le \frac{M}{N^p} \le \varepsilon.
			\end{equation}
			This implies that the sequence $(\mu_n)_{n\ge0}$ is tight, and hence is relatively compact with respect to the weak topology. 
Consequently, to find the weak compactness of $\al K_M^\kappa$, it remains to show that $\al K_M^\kappa$ is weakly closed.  Let $(\mu_n)_{n\ge0}\subset \al K_M^\kappa$ be a sequence such that $\mu_n \Rightarrow \mu$ weakly.
For any $m>0$ and $p>2$, define $\varphi_m(y):= m\wedge |y|^p$.
Since $\varphi_m \in C_{\mathrm{b}}(\bb R^r)$, the weak convergence implies that
$$
\int_{\bb R^r} (m\wedge |y|^p)\,\mu(\rrm dy)
=
\lim_{n\to\infty}
\int_{\bb R^r} (m\wedge |y|^p)\,\mu_n(\rrm dy).
$$
Since $m\wedge |y|^p \le |y|^p$, we have
$$
\int_{\bb R^r} (m\wedge |y|^p)\,\mu_n(\rrm dx)
\le
\int_{\bb R^r} |y|^p\,\mu_n(\rrm dx)
\le M.
$$
Taking $n\to+\infty$ yields
$$
\int_{\bb R^r} (m\wedge |y|^p)\,\mu(\rrm dx)
\le M,
\text{ for all } m>0.
$$
Now, since $m \wedge |y|^p \uparrow |y|^p$ pointwise as $m \to +\infty$,
the monotone convergence theorem implies
\begin{equation}\label{eqn:mubound}
    \int_{\bb R^r} |y|^p\,\mu(\rrm dx)
=
\lim_{m\to\infty}
\int_{\bb R^r} (m\wedge |y|^p)\,\mu(\rrm dx)
\le M.
\end{equation}

Next, let $v\in\bb R^r$ be an arbitrary unit vector and define $f_v(x)=v^\top xx^\top v=(v^\top x)^2$. 
Note that $\mu(|\cdot|^p)\le M$ and $\sup_{n\ge0}\mu_n(|\cdot|^p)\le M$ by assumption. 

Since $\mu_n\Rightarrow\mu$ weakly, for any $R>0$ we have $$\lim_{n\to\infty}\int_{\{x\in\bb R^r:|x|\le R\}}f_v(x)\rrm d\mu_n=\int_{\{x\in\bb R^r:|x|\le R\}}f_v(x)\rrm d\mu.$$
By Chebyshev's inequality, we also have $$\int_{\{x\in\bb R^r:|x|>R\}}f_v(x)\rrm d\mu_n\le \int_{\{x\in\bb R^r:|x|> R\}}|x|^2\rrm d\mu_n\le R^{2-p}\mu_n(|\cdot|^p)\le MR^{2-p}.$$
Similarly, 
$$\int_{\{x\in\bb R^r:|x|>R\}}f_v(x)\rrm d\mu\le \int_{\{x\in\bb R^r:|x|> R\}}|x|^2\rrm d\mu\le R^{2-p}\mu(|\cdot|^p)\le MR^{2-p}.$$
Thus, for any $\varepsilon > 0$, by choosing $R$ and $n$ sufficiently large, we obtain
\begin{equation}
\begin{aligned}
        \Big|\int_{\bb R^r}f_v(x)\rrm d\mu_n-\int_{\bb R^r}f_v(x)\rrm d\mu\Big|&\le \Big|\int_{\{x\in\bb R^r:|x|\le R\}}f_v(x)\rrm d\mu_n-\int_{\{x\in\bb R^r:|x|\le R\}}f_v(x)\rrm d\mu\Big|\\&\quad+\Big|\int_{\{x\in\bb R^r:|x|>R\}}f_v(x)\rrm d\mu_n\Big|+\Big|\int_{\{x\in\bb R^r:|x|> R\}}f_v(x)\rrm d\mu\Big|\le \varepsilon.
\end{aligned}
\end{equation}
This establishes $$\int_{\bb R^r}f_v(x)\rrm d\mu=\lim_{n\to\infty}\int_{\bb R^r}f_v(x)\rrm d\mu_n.$$
Consequently,
\begin{equation}\label{eqn:musemi}
    v^\top C_{\mu}v=\int_{\bb R^r}f_v(x)\rrm d\mu=\lim_{n\to\infty}\int_{\bb R^r}f_v(x)\rrm d\mu_n\ge\kappa.
\end{equation}
Combining Eqns.~\eqref{eqn:mubound} and \eqref{eqn:musemi}, we conclude that $\mu\in\al K_M^\kappa$.
            Additionally, for any $\mu\in\al K_M^\kappa$, Chebyshev's inequality implies
			\begin{equation}
				\int_{\{x\in\bb R^r:|x|\ge n\}}
                |x|^2\mu(\rrm dx)\le \frac{1}{n^{p-2}}\int_{x\in\bb R^r}|x|^p\mu(\rrm dx)\le \frac{M}{n^{p-2}}.
			\end{equation}
			Consequently,
            \begin{equation}\label{eqn:muinf}
                0\le \lim_{n\to\infty}\sup_{\mu\in\al K_M^\kappa}\int_{\{x\in\bb R^r:|x|\ge n\}}
                |x|^2\mu(\rrm dx)\le \lim_{n\to\infty}\frac{M}{n^{p-2}}=0.
            \end{equation}
            It follows that $\al K^\kappa_M$ is a compact convex subset.
           
			Next, we verify the second condition of Theorem~\ref{fixedpoint}, i.e., that the set $\Xi_\mu^\kappa$ is convex for every $\mu \in \mathcal K_M^\kappa$.
			
            By the definition of $\Xi_\mu^\kappa$ in Eqn.~\eqref{eqn:Xi}, for any $\nu_1, \, \nu_2 \in \Xi_\mu^\kappa$ there exist invariant probability measures $\pi^{(\mu)}_1,\, \pi^{(\mu)}_2$ associated with the same frozen semigroup $\al P_t^{(\mu)}$ whose second marginals are $\nu_1$ and $\nu_2$, respectively.
            
For any
$\alpha\in[0,1]$, the convex combination
$
\alpha\pi_1^{(\mu)}+(1-\alpha)\pi^{(\mu)}_2
$
is also 
an invariant probability measure for $\al P_t^{(\mu)}$.

Consequently,
$$
\alpha\nu_1+(1-\alpha)\nu_2\in \Xi_\mu^\kappa,
$$
which proves that $\Xi_\mu^\kappa$ is convex.
            
			Finally, we verify the third condition of Theorem~\ref{fixedpoint}, namely that the graph of $\gamma$ is closed in $\al K_M^\kappa\times \al K^\kappa_M$. 
            Let $(\mu_n)_{n\ge 1}$ and $(\nu_n)_{n\ge 1}$ be sequences in $\mathcal K_M^\kappa$ such that $\nu_n \in \gamma(\mu_n)$ for every $n$, and assume there exist $\nu,\, \mu \in \mathcal K_M^\kappa$ with $\lim_{n\to\infty} \bb W_2(\nu_n,\nu) = 0$ and $\lim_{n\to\infty} \bb W_2(\mu_n,\mu) = 0$.
            
By Eqn.~\eqref{eqn:Xi}, for each $n$ there exists a probability measure $\pi_n\in\mathscr P(V_r(\mathbb R^d)\times \bb R^r)$ such that $\pi_n\in\Lambda_{\mu_n}^\kappa$; i.e., $\pi_n$ is an invariant probability measure for the frozen semigroup $P_t^{(\mu_n)}$, and $\nu_n$ is the second marginal of $\pi_n$. 
Since $\pi_n$ is invariant for $\al P_t^{(\mu_n)}$ for each $n \ge 1$, for every bounded Lipschitz continuous function $f$ on $V_r(\mathbb R^d) \times \mathbb R^r$ we have
$$
\int f\,\rrm d\pi_n
=
\int \al P_t^{(\mu_n)}f\,\rrm d\pi_n,
\text{ for all } t\ge0.
$$

Next, we show that the sequence $\{\pi_n\}_{n\ge1}$ is tight. By Lemma~\ref{bound} and Jensen's inequality, for any $N\ge0$ and $t\ge0$, it holds that
 \begin{equation*}
     \begin{aligned}
         \pi_n(\|\cdot\|^p\wedge N)&=\bb E[\|(U_t^{(\mu_n)}, Y^{(\mu_n)}_t)\|^p\wedge N]\\&\le M+\int_{V_r(\bb R^d)\times\bb R^r}\big(\big(|y|^p\rrm e^{-\lambda t}\big)\wedge N\big)\pi_n(\rrm dx,\rrm dy),
     \end{aligned}
 \end{equation*}
 where $M$ is given in Eqn.~\eqref{eqn:M}.
 
Letting $t\to+\infty$ and applying the dominated convergence theorem, we obtain
 \begin{equation*}
     \pi_n(\|\cdot\|^p\wedge N)\le M.
 \end{equation*}
 Then, letting $N\to\infty$ and using the monotone convergence theorem, it follows that
 \begin{equation*}
     \pi_n(\|\cdot\|^p)\le M.
 \end{equation*}
By Chebyshev's inequality, for any $\varepsilon>0$, one may choose $R> 0$ sufficiently large such that 
 \begin{equation}
     \pi_n(\|(\cdot,\cdot)\|>R)\le R^{-p}\pi_n(\|\cdot\|^p)\le R^{-p}M\le \varepsilon,
 \end{equation}
 which implies that the sequence $(\pi_n)_{n\in\bb N}$ is tight.
Passing to a subsequence if necessary, we may assume
$\pi_n\Rightarrow \pi^*$ weakly for some
$\pi^*\in\mathscr P(V_r(\mathbb R^d)\times \bb R^r)$.

We will show that $\pi^*$ is invariant for $\al P^{(\nu)}_t$. Fix $t>0$ and let $f$ be bounded Lipschitz.
Then
$$
\int_{} f\,\rrm d\pi_n
-
\int \al P_t^{(\mu)}f\,\rrm d\pi_n
=
\int\big(\al P_t^{(\mu_n)}f-\al P_t^{(\mu)}f\big)\,\rrm d\pi_n.
$$
By Lemma~\ref{eq.continuous}, there exists a constant $C_t>0$ such that
$$
\big|\al P_t^{(\mu_n)}f(v,y)-\al P_t^{(\mu)}f(v,y)\big|
\le
\|f\|_{\mathrm{Lip}}\,C_t(1+|y|^2)\,\bb W_2(\mu_n,\mu),
$$
where the constant $C$ appearing in Lemma~\ref{eq.continuous} is independent of $(v,y)$.
Hence
$$
\left|
\int\big(\al P_t^{(\mu_n)}f-\al P_t^{(\mu)}f\big)\,\rrm d\pi_n
\right|
\le
\|f\|_{\mathrm{Lip}}\,C_t(1+M)\,\bb W_2(\mu_n,\mu)\to0.
$$
Since $(U_t^{(\mu)},Y_t^{(\mu)})$ is Feller by Lemma~\ref{le.feller},
$\al P_t^{(\mu)}f$ is a bounded continuous function of $V_r(\bb R^d)\times\bb R^r$. Therefore, taking $n\to\infty$ yields that
$$
\int f\,\rrm d\pi^*=\int \al P_t^{(\mu)}f\,\rrm d\pi^*.
$$
Thus $\pi^*$ is an invariant probability measure for the frozen
system corresponding to $\mu$. For any bounded continuous function $f:\bb R^r\to\bb R$, we have 
$$\int_{V_r(\bb R^d)\times\bb R^r}f(y)\pi^*(\rrm du,\rrm dy)=\lim_{n\to\infty}\int_{V_r(\bb R^d)\times\bb R^r}f(y)\pi_n(\rrm du,\rrm dy)=\lim_{n\to\infty}\int_{\bb R^r}f(y)\nu_n(\rrm dy)=\int_{\bb R^r}f(y)\nu(\rrm dy),$$which shows that
$\nu\in\gamma(\mu)$.

We have now verified all the conditions of Theorem~\ref{fixedpoint}. Hence, by the Kakutani--Fan--Glicksberg
fixed point theorem, there exists $\mu\in\mathcal K_M^\kappa$ such that
$$
\mu\in\gamma(\mu).
$$
Therefore, there exists $\pi^*\in\mathscr P(V_r(\mathbb R^d)\times\bb R^r)$ such that
$\pi^*\in\Lambda_\mu^\kappa$, i.e., $\pi^*$ is an invariant probability measure for the frozen system
with parameter $\mu$ and $\mu$ is the second marginal of $\pi^*$.

By definition of the frozen system, when the law of $Y_t^{(\mu)}$ is
equal to $\mu$, the frozen equation coincides with the original strong DO
system~\eqref{eq:DO}. Hence $\pi^*$ is an invariant probability
measure for the strong DO solution.

In summary, since $\mu\in\mathcal K_M^\kappa\subset\mathscr P_2^\kappa(\mathbb R^r)$,
we have $\det(C_\mu)\neq0$, (in particular, $C_\mu\succeq \kappa I_r$, $\kappa$ given in Eqn.~\eqref{eqn:kappa}). This completes the proof.
\end{proof}

\appendix
\section{Some property of frozen DO system}
Recall the frozen DO system mentioned in Section~\ref{sec:invariant}. For any fixed $\kappa > 0$, let $\mu \in \scr P_{2}^{\kappa}(\bb R^r)$. The dynamics are given by
\begin{equation}
             \begin{aligned}
                 C_{\mu}\rrm dU_t^{(\mu)}&=\int_{\bb R^r}x b(U_t^{(\mu)\top} x)^\top\mu(\rrm dx)(I_{d}-P_{U^{(\mu)}_t})\rrm dt;\\
				\rrm d Y_t^{(\mu)}&=U_t^{(\mu)}b(U_t^{(\mu)\top} Y^{(\mu)}_t)\rrm dt+U_t^{(\mu)}\sigma(U_t^{(\mu)\top} Y^{(\mu)}_t)\rrm dB_t.
             \end{aligned}
         \end{equation} 
\begin{theorem}\label{theo:Ymubound}
    Suppose Assumptions (H3) and (H7) hold. 
	For any initial condition $\zeta\in L^2(\Omega,V_r(\bb R^d))$ and $\xi\in L^2(\Omega,\bb R^r)$ and any $T\ge0$, there exists a constant $C(b,\sigma,T)$ such that 
    \begin{equation}
        \begin{aligned}
            &\bb E_0[\|U^{(\mu)}_t\|^2_\rrm{HS}]=\zeta;\\
            &\bb E_0[|Y^{(\mu)}_t|^2]\le C(b,\sigma,T)\big(|\xi|^2+1\big),
        \end{aligned}
    \end{equation}
    for all $t\in[0,T]$. In particular,
    \begin{equation}
        \begin{aligned}
            &\bb E[\|U^{(\mu)}_t\|^2_\rrm{HS}]=r;\\
            &\bb E[|Y^{(\mu)}_t|^2]\le C(b,\sigma,T)\big(\bb E[|\xi|^2]+1\big),
        \end{aligned}
    \end{equation}
    holds for every $t\in[0,T]$.
\end{theorem}
\begin{proof}
    In Section~\ref{sec:invariant}, we have shown that the frozen DO system admits a unique solution $(U^{(\mu)},Y^{(\mu)})$.
    Note that the function $U^{(\mu)}$ is absolutely continuous on $[0,T]$, so it is differentiable almost everywhere. The derivative $\frac{\rrm dU^{(\mu)}_t}{\rrm dt}$ satisfies 
			\begin{equation}
				\frac{\rrm dU^{(\mu)}_t}{\rrm dt}U_t^{{(\mu)}\top}=C_{\mu}^{-1}\int_{\bb R^r}x b(U_t^{{(\mu)}\top} x)^\top\mu(\rrm dx)(I_{d}-P_{U^{(\mu)}_t})U^{{(\mu)}\top}_t=O_r,
			\end{equation}
			where $O_r$ is zero matrix. Thus $\frac{\rrm d}{\rrm dt}\big(U^{(\mu)}_tU_t^{{(\mu)}\top}\big)=O_r$. Therefore, it follows that
			\begin{equation}
				\bb E_0[\|U_t^{(\mu)}\|^2_\rrm{HS}]=\bb E_0[\|U_0^{(\mu)}\|^2_\rrm{HS}]=\|\zeta\|^2_\rrm{HS}.
			\end{equation}
    Using It\^o's isometry, H\"older's inequality, together with Assumptions (H3), (H7), we obtain 
    \begin{align*}
        \bb E_0[|Y^{(\mu)}_t|^2]&\le 3|Y^{(\mu)}_0|^2+3\bb E_0\Big[\Big|\int_0^tU^{(\mu)}_sb\big(U^{(\mu)\top}_sY^{(\mu)}_s\big)\rrm ds\Big|^2\Big]+3\bb E_0\Big[\Big|\int_0^tU^{(\mu)}_s\sigma\big(U^{(\mu)\top}_sY^{(\mu)}_s\big)\rrm dB_s\Big|^2\Big]\\&\le 
        3|Y^{(\mu)}_0|^2+3t\int_0^t\bb E_0\big[\big|U^{(\mu)}_sb\big(U^{(\mu)\top}_sY^{(\mu)}_s\big)\big|^2\big]\rrm ds+3\int_0^t\bb E_0\big[\big|U^{(\mu)}_s\sigma\big(U^{(\mu)\top}_sY^{(\mu)}_s\big)\big|^2\big]\rrm ds\\
        &\le 3|Y^{(\mu)}_0|^2+C(b,\sigma,T)+C(b,\sigma,T)\int_0^t\bb E_0[|Y^{(\mu)}_s|^2]\rrm ds.
    \end{align*}
    Then by the Gr\"onwall lemma, we have 
    \begin{equation}
        \bb E_0[|Y^{(\mu)}_t|^2]\le C(b,\sigma,T)\big(|Y^{(\mu)}_0|^2+1\big)=C(b,\sigma,T)\big(|\xi|^2+1\big).
    \end{equation}
\end{proof}
\begin{theorem}\label{theo:inverse}
    Suppose that Assumptions (H4), (H5), (H7) and (H8) hold. Moreover, for any $p\ge 2$, suppose $2L_{b,7}-(p-1)L_{\sigma,1}>1$ holds, where $L_{b,7}$ and $L_{\sigma,1}$ are in (H8) and (H3), respectively. Then, for all $t \ge 0$, it holds that 
    \begin{equation}
        \bb E[Y_t^{(\mu)}Y_t^{(\mu)\top}]\succeq \rrm e^{-\frac{t}{\varepsilon}}\bb E[Y_0^{(\mu)}Y_0^{(\mu)\top}]+\varepsilon\frac{L_{\sigma,3}}{2}(1-\frac{\big(\bb E[|Y_0^{(\mu)}|^p\big)^{2/p}}{1+M^{2/p}})(1-\rrm e^{-\frac{t}{\varepsilon}})I_r,
    \end{equation}
    where $\varepsilon=\frac{L_{\sigma,3}}{2L_{b,6}(1+M^{2/p})}$, and $M$ is defined in Eqn.~\eqref{eqn:M}.
\end{theorem}
    \begin{proof}
        by It\^o's formula
    \begin{align*}
        \rrm d(Y^{(\mu)}_tY^{(\mu)\top}_t)=\rrm d(Y^{(\mu)}_t)Y^{(\mu)\top}_t+Y^{(\mu)}_t\rrm d(Y^{(\mu)\top}_t)+U^{(\mu)}_t\sigma\big(U^{(\mu)\top}_tY^{(\mu)}_t\big)\Big(U^{(\mu)}_t\sigma\big(U^{(\mu)\top}_tY^{(\mu)}_t\big)\Big)^\top\rrm dt.
    \end{align*}
    Taking the expectation of both sides yields
    \begin{align*}
        \rrm d \bb E[Y^{(\mu)}_tY^{(\mu)\top}_t]&=\bb E\big[U^{(\mu)}_tb\big(U^{(\mu)\top}_tY^{(\mu)}_t\big)Y^{(\mu)\top}_t\big]\rrm dt+\bb E[Y^{(\mu)}_t\big(U^{(\mu)}_tb\big(U^{(\mu)\top}_tY^{(\mu)}_t\big)\big)^\top]\rrm dt\\&\quad+\bb E\big[U^{(\mu)}_t\sigma\big(U^{(\mu)\top}_tY^{(\mu)}_t\big)\Big(U^{(\mu)}_t\sigma\big(U^{(\mu)\top}_tY^{(\mu)}_t\big)\Big)^\top\big]\rrm dt.
    \end{align*}
    Then, for any unit vector $v\in\bb R^r$, we have 
    \begin{equation}
        v^\top\frac{\rrm d \bb E[Y^{(\mu)}_tY^{(\mu)\top}_t]}{\rrm dt}v=2v^\top\bb E[Y^{(\mu)}_t\big(U^{(\mu)}_tb\big(U^{(\mu)\top}_tY^{(\mu)}_t\big)\big)^\top]v+v^\top\bb E\big[U^{(\mu)}_t\sigma\big(U^{(\mu)\top}_tY^{(\mu)}_t\big)\Big(U^{(\mu)}_t\sigma\big(U^{(\mu)\top}_tY^{(\mu)}_t\big)\Big)^\top\big]v.
    \end{equation}
    By (H5), the last term can be bounded below as
    \begin{align}\label{eqn:sigmav}
        v^\top\bb E\big[U^{(\mu)}_t\sigma\big(U^{(\mu)\top}_tY^{(\mu)}_t\big)\Big(U^{(\mu)}_t\sigma\big(U^{(\mu)\top}_tY^{(\mu)}_t\big)\Big)^\top\big]v\ge L_{\sigma,3}\bb E\big[v^\top U^{(\mu)}_tU^{(\mu)\top}_t v\big]\ge L_{\sigma,3}.
    \end{align}
    By H\"older's inequality, (H7) and Lemma~\ref{bound}, the first term can be bounded as 
    \begin{align*}
        \Big|v^\top\bb E[Y^{(\mu)}_t\big(U^{(\mu)}_tb\big(U^{(\mu)\top}_tY^{(\mu)}_t\big)\big)^\top]v\Big|&\le \Big|\bb E[v^\top Y^{(\mu)}_tb^\top\big(U^{(\mu)\top}_tY^{(\mu)}_t\big)v]\Big|\\&\le 
        \frac{1}{2\varepsilon}\bb E\big[\big|v^\top Y^{(\mu)}_t\big|^2\big]+\frac{\varepsilon}{2}\bb E\big[\big|b\big(U^{(\mu)\top}_tY^{(\mu)}_t\big)\big|^2\big]\\&\le 
        \frac{1}{2\varepsilon}\bb E\big[v^\top Y^{(\mu)}_tY^{(\mu)\top}_tv\big]+\frac{\varepsilon}{2}L_{b,6}\big(1+\bb E[|Y^{(\mu)}_t|^2]\big)\\&\le 
        \frac{1}{2\varepsilon}\bb E\big[v^\top Y^{(\mu)}_tY^{(\mu)\top}_tv\big]+\frac{\varepsilon}{2}L_{b,6}\Big(1+M^{2/p}+\big(\bb E[|Y_0^{(\mu)}|^p\big)^{2/p}\Big),\text{ for any }\varepsilon>0.
    \end{align*}
    Taking $\varepsilon=\frac{L_{\sigma,3}}{2L_{b,6}(1+M^{2/p})}$ yields that 
    \begin{align*}
         v^\top\frac{\rrm d \bb E[Y^{(\mu)}_tY^{(\mu)\top}_t]}{\rrm dt}v&\ge -\frac{1}{\varepsilon}\bb E\big[v^\top Y^{(\mu)}_tY^{(\mu)\top}_tv\big]+\frac{L_{\sigma,3}}{2}\Big(1-\frac{\big(\bb E[|Y_0^{(\mu)}|^p\big)^{2/p}\Big)}{1+M^{2/p}}\Big),
    \end{align*}
    from which we deduce 
    \begin{align*}
        v^\top\frac{\rrm d \rrm e^\frac{t}{\varepsilon}\bb E[Y^{(\mu)}_tY^{(\mu)\top}_t]}{\rrm dt}v&\ge \frac{L_{\sigma,3}}{2}\Big(1-\frac{\big(\bb E[|Y_0^{(\mu)}|^p\big)^{2/p}}{1+M^{2/p}}\Big)\rrm e^\frac{t}{\varepsilon}.
    \end{align*}
    From the previous estimate, we obtain
    \begin{equation}\label{eq:Ysemi}
    \begin{aligned}
        v^\top\bb E[Y^{(\mu)}_tY^{(\mu)\top}_t]v&\ge  \rrm e^{-\frac{t}{\varepsilon}}v^\top\bb E[Y^{(\mu)}_0Y^{(\mu)\top}_0]v+\frac{L_{\sigma,3}}{2}\Big(1-\frac{\big(\bb E[|Y_0^{(\mu)}|^p\big)^{2/p}}{1+M^{2/p}}\Big)\int_0^t\rrm e^{-\frac{s}{\varepsilon}}\rrm ds\\&=
        \rrm e^{-\frac{t}{\varepsilon}}v^\top\bb E[Y^{(\mu)}_0Y^{(\mu)\top}_0]v+\varepsilon\frac{L_{\sigma,3}}{2}\Big(1-\frac{\big(\bb E[|Y_0^{(\mu)}|^p\big)^{2/p}}{1+M^{2/p}}\Big)(1-\rrm e^{-\frac{t}{\varepsilon}}).        
    \end{aligned}
    \end{equation}
    \end{proof}
	\bibliographystyle{plain}

\begin{thebibliography}{10}

\bibitem{aliprantis2006infinite}
Charalambos~D Aliprantis and Kim~C Border.
\newblock {\em Infinite dimensional analysis: a hitchhiker’s guide}.
\newblock Springer, 2006.

\bibitem{ambrosio2005gradient}
Luigi Ambrosio, Nicola Gigli, and Giuseppe Savar{\'e}.
\newblock {\em Gradient flows: in metric spaces and in the space of probability measures}.
\newblock Springer, 2005.

\bibitem{ash2000probability}
Robert~B Ash and Catherine~A Dol{\'e}ans-Dade.
\newblock {\em Probability and measure theory}.
\newblock Academic press, 2000.

\bibitem{bao2022existence}
Jianhai Bao, Michael Scheutzow, and Chenggui Yuan.
\newblock Existence of invariant probability measures for functional mckean-vlasov sdes.
\newblock {\em Electronic Journal of Probability}, 27:1--14, 2022.

\bibitem{coupling}
Jianhai Bao and Jian Wang.
\newblock Coupling methods and exponential ergodicity for two-factor affine processes.
\newblock {\em Math. Nachr.}, 296(5):1716--1736, 2023.

\bibitem{basak2016langevin}
Gopal~K Basak and Arunangshu Biswas.
\newblock Langevin type limiting processes for adaptive mcmc.
\newblock {\em Indian Journal of Pure and Applied Mathematics}, 47(2):301--328, 2016.

\bibitem{blomker2022continuous}
Dirk Bl{\"o}mker, Claudia Schillings, Philipp Wacker, and Simon Weissmann.
\newblock Continuous time limit of the stochastic ensemble kalman inversion: strong convergence analysis.
\newblock {\em SIAM Journal on Numerical Analysis}, 60(6):3181--3215, 2022.

\bibitem{charous2023dynamically}
Aaron Charous and Pierre~FJ Lermusiaux.
\newblock Dynamically orthogonal runge--kutta schemes with perturbative retractions for the dynamical low-rank approximation.
\newblock {\em SIAM Journal on Scientific Computing}, 45(2):A872--A897, 2023.

\bibitem{MR2091955}
Mu-Fa Chen.
\newblock {\em From {M}arkov chains to non-equilibrium particle systems}.
\newblock World Scientific Publishing Co., Inc., River Edge, NJ, second edition, 2004.

\bibitem{chikuse2003statistics}
Yasuko Chikuse.
\newblock {\em Statistics on special manifolds}, volume 174.
\newblock Springer Science \& Business Media, 2003.

\bibitem{da1996ergodicity}
Giuseppe Da~Prato and Jerzy Zabczyk.
\newblock {\em Ergodicity for infinite dimensional systems}, volume 229.
\newblock Cambridge university press, 1996.

\bibitem{fan1952fixed}
Ky~Fan.
\newblock Fixed-point and minimax theorems in locally convex topological linear spaces.
\newblock {\em Proceedings of the National Academy of Sciences}, 38(2):121--126, 1952.

\bibitem{feppon2018dynamically}
Florian Feppon and Pierre~FJ Lermusiaux.
\newblock Dynamically orthogonal numerical schemes for efficient stochastic advection and lagrangian transport.
\newblock {\em Siam Review}, 60(3):595--625, 2018.

\bibitem{normaldis}
Clark~R. Givens and Rae~Michael Shortt.
\newblock A class of {W}asserstein metrics for probability distributions.
\newblock {\em Michigan Math. J.}, 31(2):231--240, 1984.

\bibitem{guilleminot2014ito}
Johann Guilleminot and Christian Soize.
\newblock It\^o sde--based generator for a class of non-gaussian vector-valued random fields in uncertainty quantification.
\newblock {\em SIAM Journal on Scientific Computing}, 36(6):A2763--A2786, 2014.

\bibitem{horn2012matrix}
Roger~A Horn and Charles~R Johnson.
\newblock {\em Matrix analysis}.
\newblock Cambridge university press, 2012.

\bibitem{karatzas2014brownian}
Ioannis Karatzas and Steven Shreve.
\newblock {\em Brownian motion and stochastic calculus}.
\newblock springer, 2014.

\bibitem{kazashi2021existence}
Yoshihito Kazashi and Fabio Nobile.
\newblock Existence of dynamical low rank approximations for random semi-linear evolutionary equations on the maximal interval.
\newblock {\em Stochastics and Partial Differential Equations: Analysis and Computations}, 9(3):603--629, 2021.

\bibitem{DLRASDE}
Yoshihito Kazashi, Fabio Nobile, and Fabio Zoccolan.
\newblock Dynamical low-rank approximation for stochastic differential equations.
\newblock {\em Math. Comp.}, 94(353):1335--1375, 2025.

\bibitem{kazashi2026numerical}
Yoshihito Kazashi, Fabio Nobile, and Fabio Zoccolan.
\newblock Numerical methods for dynamical low-rank approximations of stochastic differential equations--part i: Time discretization.
\newblock {\em arXiv preprint arXiv:2601.21428}, 2026.

\bibitem{khasminskii2011stochastic}
Rafail Khasminskii.
\newblock {\em Stochastic stability of differential equations}, volume~66.
\newblock Springer Science \& Business Media, 2011.

\bibitem{koch2007dynamical}
Othmar Koch and Christian Lubich.
\newblock Dynamical low-rank approximation.
\newblock {\em SIAM Journal on Matrix Analysis and Applications}, 29(2):434--454, 2007.

\bibitem{llopis2018particle}
Francesc~Pons Llopis, Nikolas Kantas, Alexandros Beskos, and Ajay Jasra.
\newblock Particle filtering for stochastic navier--stokes signal observed with linear additive noise.
\newblock {\em SIAM Journal on Scientific Computing}, 40(3):A1544--A1565, 2018.

\bibitem{musharbash2018dual}
Eleonora Musharbash and Fabio Nobile.
\newblock Dual dynamically orthogonal approximation of incompressible navier stokes equations with random boundary conditions.
\newblock {\em Journal of Computational Physics}, 354:135--162, 2018.

\bibitem{musharbash2015error}
Eleonora Musharbash, Fabio Nobile, and Tao Zhou.
\newblock Error analysis of the dynamically orthogonal approximation of time dependent random pdes.
\newblock {\em SIAM Journal on Scientific Computing}, 37(2):A776--A810, 2015.

\bibitem{niknam2024nuclear}
Mohamad Niknam and Louis-S Bouchard.
\newblock Nuclear induction line shape: Non-markovian diffusion with boundaries.
\newblock {\em The Journal of Chemical Physics}, 160(2), 2024.

\bibitem{nobile2026dynamicallowrankensemblekalman}
Fabio Nobile, Sébastien Riffaud, and Thomas~Trigo Trindade.
\newblock Dynamical low-rank ensemble kalman filter for state/parameter estimation.
\newblock {\em ArXiv}, \url{https://arxiv.org/abs/2602.06614}, 2026.

\bibitem{nobile2025dynamicallowrankapproximationskalman}
Fabio Nobile and Thomas~Trigo Trindade.
\newblock Dynamical low-rank approximations for kalman filtering.
\newblock {\em ArXiv}, \url{https://arxiv.org/abs/2509.11210}, 2025.

\bibitem{pavliotis2014stochastic}
Grigorios~A Pavliotis.
\newblock Stochastic processes and applications.
\newblock {\em Texts in applied mathematics}, 60:41--43, 2014.

\bibitem{sapsis2009dynamically}
Themistoklis~P Sapsis and Pierre~FJ Lermusiaux.
\newblock Dynamically orthogonal field equations for continuous stochastic dynamical systems.
\newblock {\em Physica D: Nonlinear Phenomena}, 238(23-24):2347--2360, 2009.

\bibitem{ueckermann2013numerical}
Mattheus~P Ueckermann, Pierre~FJ Lermusiaux, and Themistoklis~P Sapsis.
\newblock Numerical schemes for dynamically orthogonal equations of stochastic fluid and ocean flows.
\newblock {\em Journal of Computational Physics}, 233:272--294, 2013.

\end{thebibliography}

\end{document}